\documentclass[11pt]{article}
\pdfoutput=1
\usepackage{enumerate}
\usepackage[OT1]{fontenc}
\usepackage{amsmath,amssymb}
\usepackage[usenames]{color}
\usepackage{diagbox}
\usepackage[numbers]{natbib}
\usepackage{smile}
\usepackage{multirow}
\usepackage[colorlinks,
            linkcolor=red,
            anchorcolor=blue,
            citecolor=blue
            ]{hyperref}

\def\tr{\mathop{\text{tr}}\kern.2ex}

\def\P{{\mathbb P}}
\def\E{{\mathbb E}}

\def\V{{\mathbb V}}

\long\def\comment#1{}

\def\tr{\mathop{\text{Tr}}}

\newcommand{\bel}{\begin{eqnarray}\label}
\newcommand{\eel}{\end{eqnarray}}
\newcommand{\bes}{\begin{eqnarray*}}
\newcommand{\ees}{\end{eqnarray*}}

\newcommand{\normmm}[1]{{\left\vert\kern-0.25ex\left\vert\kern-0.25ex\left\vert #1 
    \right\vert\kern-0.25ex\right\vert\kern-0.25ex\right\vert}}

\usepackage{mathrsfs}

\usepackage{fullpage}

\def\##1\#{\begin{align}#1\end{align}}
\def\$#1\${\begin{align*}#1\end{align*}}
\usepackage{enumitem}

\usepackage[T1]{fontenc}
\usepackage[utf8]{inputenc}
\usepackage{authblk}

\title{ $\epsilon $-Strong Simulation of Fractional Brownian Motion and Related Stochastic Differential Equations}
\author[1]{Yi Chen}
\author[2]{Jing Dong}
\author[3]{Hao Ni}
\affil[1]{Northwestern University, yichen2016@u.northwestern.edu}
\affil[2]{Columbia University, jing.dong@gsb.columbia.edu}
\affil[3]{University College London, h.ni@ucl.ac.uk}

\date{}
\begin{document}
\maketitle
\begin{abstract}
Consider the fractional Brownian Motion (fBM) $B^H=\{B^H(t): t \in [0,1] \}$ with Hurst index $H\in (0,1)$. 
We construct a probability space supporting both $B^H$ and a fully simulatable process $\hat B_{\epsilon}^H $ such that 
\$
\sup_{t\in [0,1]}|B^H(t)-\hat B_{\epsilon}^H(t)| \le \epsilon
\$
with probability one for any user specified error parameter $\epsilon>0$. When $H>1/2$, we further enhance our error guarantee 
to the $\alpha$-H\"older norm for any $\alpha \in (1/2,H)$. This enables us to extend our algorithm to the 
simulation of fBM driven stochastic differential equations $Y=\{Y(t):t \in[0,1]\}$. 
Under mild regularity conditions on the drift and diffusion coefficients of $Y$, 
we construct a probability space supporting both $Y$ and a fully simulatable process $\hat Y_{\epsilon}$ such that 
\$
\sup_{t\in [0,1]}|Y(t)-\hat Y_{\epsilon}^H(t)| \le \epsilon
\$
with probability one. Our algorithms enjoy the tolerance-enforcement feature, i.e., the error bounds can be updated sequentially. 
Thus, the algorithms can be readily combined with other simulation techniques like multilevel Monte Carlo to estimate expectation of 
functionals of fBMs efficiently. 
\end{abstract}
\section{Introduction}
The fractional Brownian motion (fBM) of Hurst parameter $H\in(0,1)$, $\{B^H(t): t\ge 0\}$, is a centered real-valued Gaussian process with covariance function
\# \label{covfun}
r(s,t):=E[B^H(s)B^H(t)]=\frac{1}{2}\left(|s|^{2H}+|t^{2H}-|s-t|^{2H}\right).
\#
When $H=1/2$, fBM is a Brownian motion (BM) which has independent increments. 
When $H<1/2$, the increments of fBM are negatively correlated.
In contrast, when $H>1/2$, fBM has positively correlated increments and displays long-range dependence.
Generally, fBM can be viewed as extension of BM. As it allows the increment to be correlated, it has 
been applied in communications engineering \citep{norros1995use}, 
biology and physics \citep{hofling2013anomalous}. 
See also \cite{laskin2000fractional} and \cite{nualart2006fractional} for applications in finance and turbulence.
Due to the correlated increments (lack of Markovian structure), very few closed-form expressions are known for 
performance measures related to functionals of fBMs. 
In this context, simulation-based numerical method has been a powerful tool to conduct performance analysis
for fBM driven processes.

In this paper, we develop a new class of algorithms to construct paths of fBM and fBM driven stochastic differential equations (SDEs)
with strong error guarantees. 
In particular, the algorithm allows us construct a probability space supporting both a fBM 
and a fully simulatable paths $\hat B_{\epsilon}^H$, such that 
$$\sup_{0\leq t\leq 1} |B^H(t)-\hat B_{\epsilon}^H(t)| \leq \epsilon \mbox{ w.p. $1$.}$$
Moreover, when $H>1/2$, for $\alpha\in(1/2,H)$, we can further construct
a fully simulatable paths $\hat B_{\epsilon(\alpha)}^H$, such that 
$$\sup_{0<s\leq t\leq1}\frac{\bigl|B^H(t)-\hat B_{\epsilon(\alpha)}^H(t)-(B^H(s)-\hat B_{\epsilon(\alpha)}^H(s))\bigr|}{|t-s|^{\alpha}}\leq
\epsilon \mbox{ w.p. $1$.}$$
For $H>1/2$, the control of the $\alpha$-H\"older norm allows us to use the rough-path theory \cite{lyons1998differential}
to construct a probability space supporting both a fBM-driven SDE
$$dY(t)=\mu(Y(t))dt + \sigma(Y(t))dB^H(t)$$
and a sequence of fully simulatable path $\hat Y_{\epsilon}$, such that 
$$\sup_{0\leq t\leq 1} |Y(t)-\hat Y_{\epsilon}(t)| \leq \epsilon \mbox{ w.p. $1$.}$$

In addition to the strong error guarantee, the framework we developed also enjoys the {\em tolerance-enforcement} feature.
Specifically, for any sequence $0<\epsilon_n<\epsilon_{n-1} < \dots<\epsilon_1$, we can adaptively simulate
$X_{\epsilon_n}$ conditional on $X_{\epsilon_1}, \dots, X_{\epsilon_{n-1}}$.
The tolerance-enforcement allows us to easily combine our procedure with other simulation techniques such as Multilevel
Monte Carlo (MLMC) for efficient estimations of expectations \cite{giles2008multilevel}, 
and various randomization techniques to remove estimation bias (see for example \cite{rhee2012new,beskos2012varepsilon}).
In this paper, we provide a concrete demonstration as how to combine our algorithm with MLMC.
The strong error bound provides the extra benefit of simplifying the rate of convergence analysis and complexity analysis in this case. 

In terms of the computational complexity, our algorithm achieves the near optimal complexity.
Specifically, for fBM, to achieve an $\epsilon$ error bound, the computational cost is $O_p(\epsilon^{-1/(H-\delta)})$ for any $\delta>0$. 
When $H>1/2$, for fBM driven SDEs, under suitable regularity conditions on the drift and diffusion terms,
to achieve an $\epsilon$ error bound, the computation cost is $O_p(\epsilon^{-1/(2\alpha-1)})$ for any $\alpha\in(1/2,H)$. 
When combined with MLMC for expectation estimation, to achieve an $\epsilon^2$ mean squared error (MSE), 
for Lipschitz continuous functionals of fBM, we are able to reduce the computational complexity to the canonical $O_p(\epsilon^{-2}\log(1/\epsilon))$ when $H>1/2$.
For fBM driven SDEs, we are able to reduce the computational complexity to the canonical $O_p(\epsilon^{-2}\log(1/\epsilon))$ when $\alpha>3/4$.

The simulation framework we developed in this paper is an important extension of the framework developed in \cite{blanchet2017varepsilon} for BM driven SDEs,
and is the very first of its kind for fBM and fBM driven SDEs (with almost sure error bound and tolerance-enforcement feature). In the process of developing the simulation algorithms, we also extend existing results and prove new results
about properties of fBM. These results may be of independent interests to the analysis of fBM. 
In particular, for midpoint displacement decomposition (wavelet expansion using the Harr wavelets) of fBM,
we establish its convergence rate in the uniform norm for $H\in (0,1)$,
and in the $\alpha$-H\"older norm for $H\in (1/2,1)$, $\alpha<H$ almost surely.
These results rely on detailed analysis of the decay rate of the conditional variance of fBM at different dyadic levels. 
For fBM driven SDEs, we provide explicit characterization of the constant term for the error induced by Euler scheme. 
This extends previous results on the convergence rate of Euler scheme in path-by-path construction of fBM driven SDEs. 

Throughout the paper, we denote \$\|u\|_{\infty}:=\sup_{0\le t\le1 }|u(t)|\ \text{and}\ \|u\|_{\alpha}:=\sup_{0\le s< t\le1 } \frac{|u(s)-u(t)|}{|s-t|^{\alpha}}\$ as the supremum norm and $\alpha$-H\"older norm of a generic function $u$ on $[0,1]$ respectively.
For $\bv$, which can be a vector, a matrix, or a tensor, we use $\| \bv\| $ to denote the maximum of the absolute value of its entries.\\

The rest of the paper is organized as follows. We conclude this section with a brief review of the literature to put our work in the right context.
We introduce the basic idea of our algorithmic development in Section \ref{sec:basic}. We then introduce properties of 
the midpoint displacement decomposition of fBM in Section \ref{sec:mid}. This provides the theoretical basis for the construction of our algorithm. The details of the actual simulation algorithm for fBM are provided in Section \ref{sec:alg}. 
In Section \ref{sec:sde}, we extend the algorithm to SDEs driven by fBM with Hurst index $H>1/2$.
We explain how our algorithm can be combined with MLMC in Section \ref{sec:mlmc}.
Lastly, in Section \ref{sec:num} we conduct some numerical experiments as a sanity check of our development 
and provide some comments about implementations.
All of the proofs of the technical lemmas are delayed until the Appendix.


\subsection{Literature review}
Our work is closely related to the line of research on {\em simulation of fBM}.
Existing methods for simulating fBM can be divided into two categories, exact method and approximation method (see \cite{dieker2004simulation} for a detailed survey). 

The exact methods aim to generate the fBM at a fixed finite set of time points from the exact distribution. 
To carry out this task efficiently is highly nontrivial, due to the correlation structure of fBM. 
Naive implementation using the Cholesky decomposition has a complexity of $O(n^3)$ for $n$ points in general.
More efficient methods have been developed in the literature, mostly of them utilizing the stationarity of the process.
For example, the Hosking method (also known as Durbin or Levinson method) (e.g., \cite{levinson1946wiener},\cite{durbin1960fitting}) generates sample path recursively, which avoids calculating the inverse of the covariance matrix. Its computational complexity is  $O(n^2)$ for a set of size $n$. The circulant embedding method, which is originally proposed by \cite{davies1987tests} and later generalized by \cite{dietrich1997fast} and \cite{wood1994simulation}. The basic idea is to find the square root decomposition of the covariance matrix by embeding the covariance matrix in a so-called circulant covariance matrix. This method can further reduce the complexity to  $O(n\log(n))$. 
Our method aim to recover the whole fBM path, but it builds on being able to generate the fBM at finite set of points, dyadic points in particular, exactly.
Thus, some of these techniques developed in the literature will be incorporated into our algorithmic development. 
However, as our algorithm also relies on sequentially update the set of points at finer and finer scales, we lose some of the stationarity structure. 

The approximation methods aim to generate approximations of the fBM sample path. The (conditionalized) random midpoint displacement (RMD) method generates the sample path recursively in a carefully designed order \citep{lau1995self,norros1999simulation}. When generating the next sample, the (conditionalized) RMD method speeds up by using only partial samples generated, instead of whole history. It achieves a computational complexity of order $O(n)$ for $n$ points, 
but the path we constructed may lose certain properties of the original sample path (e.g. long-range dependence) and it is not clear what error guarantee we would be able to achieve here. 
Some approximation methods build on special representations of fBM. For example, \cite{mandelbrot1968fractional} represents the fBM as a stochastic integral with respect to ordinary BM and approximates the integral via Riemann sum. 
Other representations of the fBM that are used to develop approximation algorithms through truncation include wavelet decompositions (see for example \cite{abry1996wavelet}, \cite{meyer1999wavelets}, and \cite{ayache2003approximating}),
spectral decompositions (see for example \cite{dieker2003spectral}, \cite{paxson1997fast} and \cite{dzhaparidze2004series}).
Our method also builds a suitable infinite series decomposition of the fBM. However, instead of applying a deterministic truncation level,
our truncation level is adapted to the sample path, i.e. random.
Our work extends this line of literature by achieving a stronger error guarantee. 
We also note that like a lot of the approximation methods developed in the literature, 
our algorithm is for pre-specified fixed time horizons.

The simulation framework we developed is also closely related to recent development in {\em $\epsilon$-strong simulation}.
The $\epsilon$-strong simulation refers to constructing a fully simulatable path whose deviation from the true path is uniformly bounded by $\epsilon$ with probability one.
\cite{beskos2012varepsilon} is among the first to develop the concept $\epsilon$-strong simulation. 
In \cite{beskos2012varepsilon}, the authors develop an $\epsilon$-strong simulation algorithm for Brownian motion.  
\cite{blanchet2013steady} and \cite{blanchet2017varepsilon} later extend the framework to reflected Brownian motion and multidimensional stochastic differential equations respectively. One important application of $\epsilon$-strong simulation algorithm is to build unbiased estimators for expectations involving functionals of the sample path \citep{beskos2012varepsilon}, or to build exact simulation algorithm for the corresponding stochastic processes (at a finite collection of time points) \citep{beskos2005exact,chen2013localization}.
\cite{blanchet2017exact} extends the algorithm developed in \cite{blanchet2017varepsilon} to construct exact simulation algorithm for multidimensional SDEs. \cite{pollock2013exact} considers the SDEs with jumps and provides a comprehensive discussion on $\epsilon$-strong and exact simulation. See also \cite{glynn2016exact} for an extension review of recent development in exact simulation and unbiased estimation algorithms. Our work contributes to this line of work by extending the $\epsilon$-strong simulation framework to fBM and fBM driven SDEs.

In terms of the methodology. Our development builds on the idea of {\em record-breakers}.
This idea was first introduced in \cite{blanchet2011exact} for exact sampling of stochastic perpetuities.
Later similar ideas have been applied to exact simulation of queueing models in steady-state \citep{blanchet2015perfect, blanchet2013steady}, and max-stable processes and related random fields \citep{liu2016optimal}. Our algorithm also builds on idea of {\em Bernoulli factory} \citep{huber2016nearly},
but in the actual implementation, we avoid using the Bernoulli factory by 
applying some properly constructed change-of-measures.

An important tool in developing $\epsilon$-strong simulation algorithm for fBM driven SDEs is the {\em rough path theory} \citep{lyons1998differential}.
The rough path theory provides us with a path-by-path construction of the SDEs. By lifting a driving signal to a rough path, the mapping from a driving rough path to the solution to the SDE is uniformly continuous under the $p$-variation metric.
Most related to our setting, \cite{coutin2002stochastic} construct a geometric rough path associated with fBM where the Hurst index $H>1/4$ and develops a Skohorod integral representation of the geometric rough path. 
More recently, \cite{nualart2011construction} develop the construction of the rough path above fBM using Volterra's representation for any $H\in (0,1)$.

There are also works analyzing the discretization error for fBM driven SDEs.
For example, \cite{mishura2008stochastic} investigates the rate of convergence of the Euler scheme. 
\cite{neuenkirch2007exact} conduct convergence analysis of a few different discretization schemes.
\cite{deya2012milstein}  propose a modified Milstein scheme without using the actual L\'evy area. More recently, \cite{hu2016rate} introduce a modified Euler scheme that works well when $H$ approaches $1/2$. 
Most previous works focus on weaker error bound than what is established in our work.
Using rough path theory, we are able to study the discretization error in a path-by-path sense.
Similar techniques are used in \cite{davie2008differential, blanchet2017varepsilon}.
In this paper, we focus on the Euler scheme to demonstrate the basic idea of the $\epsilon$-strong framework. 
We view the exploration of more sophisticated 
discretization scheme as an interesting future research direction.

Our $\epsilon$-strong simulation algorithm can be combined with MLMC. MLMC is first proposed in \cite{giles2008multilevel} to reduce the computational complexity for the expectation estimation of SDEs (driven by BM) via Euler scheme. MLMC use the multigrids ideas and has $O(\epsilon^{-2}(\log(\epsilon))^2)$ computational complexity to achieve a MSE of $O(\epsilon^2)$, which is a significant improvement from the naive Monte Carlo method. This idea is further enhanced in \cite{giles2008improved} by combing with Milstein scheme. The idea of MLMC are also extended to the estimation of functionals of more general stochastic processes. For example, \cite{dereich2011multilevel} proposes a MLMC algorithm for L\'evy-driven SDEs and \cite{bayer2016rough} extends this idea for SDEs driven by general Gaussian noise using the rough path theory. 

\section{Basic idea} \label{sec:basic}
We start by introducing the basic idea of our algorithmic development.
Recall that fBM $B^H$ is a centered real-valued Gaussian process with the covariance function given in \eqref{covfun}. By Kolmogorov continuity theorem, fBM has a continuous modification. Moreover, for any $\alpha\in(0,H)$ and $T>0$, this modification
is $\alpha$-H\"older continuous on $[0,T]$. In this paper, we refer to the modification as the fBM, and 
focus on a finite time interval $[0,1]$ with $B^H(0)=0$. 

The algorithmic development consists of two steps. First, we identify an infinite series expansion of fBM, i.e. 
\begin{equation} \label{eq:series}
B^H(t)=\sum_{k=0}^{\infty}\Lambda_{k}(t)W_{k}
\end{equation}
where $\Lambda_{k}$'s are a sequence of basis functions, and $W_k$'s are the random coefficients.
We then develop an algorithm to truncate the infinite sum up to a finite but random level, $K$,
so that the error induced by the truncated terms is suitably controlled, e.g.
$$\sup_{0\leq t\leq 1}\biggl| \sum_{k=0}^{K}\Lambda_{k}(t)W_{k}-B^H(t)\biggr|\leq \epsilon \mbox{ w.p. 1}$$

In terms of the infinite series expansion for fBM, several of them are developed in the literature, which are
based on wavelet decomposition (multi-resolution framework) (see for example \cite{meyer1999wavelets}) or  
Karhunen-Lo\'eve type of expansion (spectral theory) (see for example \cite{dzhaparidze2004series}).
Our consideration here is twofold. First, the infinite series expansion needs to converge fast in an almost sure sense 
under uniform norm or even the $\alpha$-H\"older norm. Second, the corresponding simulation algorithm can be implemented efficiently. 

In this paper, our main development follows L\'evy's midpoint displacement technique, which corresponds to the wavelet decomposition using the Haar wavelets. The challenge here is that when $H\neq 1/2$, the coefficients are correlated. 
We shall provide more analysis about the random coefficient terms in Section \ref{sec:mid}.
The actual midpoint displacement construction goes as follows. Let $D_n$ be the dyadic discretization of order $n$ and $\Delta_n$ be the mesh of the discretization. Specifically, 
\$D_n=\{t^n_0,t^n_1,\cdots,t^n_{2^n}\},\ \text{where}\ t^n_i=i/2^n, i=0,1,\cdots,2^n,\ \text{and}\ \Delta_n=1/2^n. \$
We use $\bB^H_n=(B^H(t^n_0),\ldots,B^H(t^n_{2^n}))$ to denote the value of fBM at discretization level $n$. Given a realization of $\bB^H_n$, we can construct a continuous path $B^H_n$ over time interval $[0,1]$ via linear interpolation and we call $B^H_n$ a dyadic discretization of fBM of level $n$. We notice that $ B^H_n(t)-B^H_{n-1}(t) $ has zero-value on $D_{n-1}$. At the augmented points $D_{n}/D_{n-1}= \{t^{n}_{2i+1}\}_{i=0,\cdots,2^{n-1}-1},$ where $ {i=0,1,\cdots,2^n-1} $,   we have
\$ 
B^H_n\bigl(t^{n}_{2i+1}\bigr)-B^H_{n-1}\bigl(t^{n}_{2i+1}\bigr)=B^H(t^{n}_{2i+1})-\frac{1}{2}\bigl(B^H(t^{n-1}_i)+B^H(t^{n-1}_{i+1})\bigr) 
\$ 
This is what we refer to as the midpoint displacement. In Section \ref{sec:mid}, we show that the following infinite series representation is valid
\begin{equation}\label{eq:main}
B^H(t)=\sum_{k=0}^{\infty}\left[B^H_{k}(t)-B^H_{k-1}(t)\right],
\end{equation}
i.e. $B_n^H$ converges to $B^H$ almost surely {under the supremum norm} (Theorem \ref{dyadic_con}). Here $B^H_{-1}(t)$ denotes a zero-valued constant function. Note that when we truncate the infinite series at level $n$, we obtain $B^H_n$. We also establish the rate of convergence for $B_n^H$ (Theorem \ref{ucr} and Theorem \ref{hcr}). Specifically,   
$$\| B_n^H-B^H\|_{\infty}=O\left(2^{-(H-\delta)n}\right) \mbox{ for any $\delta\in(0,H)$}.$$ 
Notice that when the $W_k$'s in the series representation \ref{eq:series} are uncorrelated, \cite{kuhn2002optimal} shows that the optimal rate of convergence, under the $L_{2}$ norm, of such series representation of fBM is $O(n^{-H}(1+\log n)^{1/2})$.
The midpoint displacement representation achieves almost the same rate of convergence \footnote{The $n$-th dyadic level involves $2^n$ time points (Gaussian random variables).}.

We next introduce the algorithmic development to truncate the infinite sum. Our goal here is to control the error of the infinite truncated terms. To achieve this, we adopt the strategy of \textquotedblleft record-breakers \textquotedblright \citep{blanchet2017varepsilon}. The general idea of record-breakers is to define a sequence of events called record-breakers, which satisfies the
following two conditions\\

\begin{itemize}
\item[C1)] The following event happens with probability one: beyond some
random but finite level, there will be no more record breakers; 
\item[C2)] By knowing that there are no more record breakers, the contribution of the infinite remaining terms are well under control.\\
\end{itemize}

In our case, we say a record is broken at level $n$ if
$$\|B^H_{n}-B^H_{n-1}\|_{\infty}\ge \rho \cdot 2^{-(H-\delta)n}.$$
{Here $\delta \in (0,H) $ and $\rho>0$ are tuning parameters, which will be specified in Theorem \ref{tail} of Section \ref{sec:mid}. The choices of these parameters will affect the efficiency of the implementation of our algorithm.}
We also denote $N$ as the level of the last record-breaker, i.e.
$$N=\sup\left\{n\geq 1: \|B^H_{n}-B^H_{n-1}\|_{\infty}\ge \rho \cdot 2^{-(H-\delta)n}\right\}.$$
For C1), we show in Theorem \ref{ucr} that $N$ has a finite moment generating function.
For C2), we notice that for $n\geq N$, we have
$$\biggl \| \sum_{k=n+1}^{\infty}\left[B^H_{k}-B^H_{k-1}\right] \biggr\|_{\infty} \leq\rho \cdot \sum_{k=n+1}^{\infty}2^{-(H-\delta)k}.$$
Thus, once we know the time of the last record-breaker $N$, to achieve a certain accuracy $\epsilon$, we just need to find $N(\epsilon)>N$, such that
$$\rho\cdot \sum_{k=N(\epsilon)+1}^{\infty}2^{-(H-\delta)k}<\epsilon,$$
then $\| B^H_{N(\epsilon)}-B^H\|_{\infty}<\epsilon$. The error bound is achieved in a path-by-path sense. In addition, we show in Theorem \ref{hcr} that conditional on $N$, we also have an explicit upper
bound for the $\alpha$-H\"older norm of $B^H$ in a path-by-path sense. This is important to develop the $\epsilon$-strong simulation algorithm for fBM driven SDEs as outlined in Section \ref{sec:sde}.

The remaining task is to find the last record-breaker $N$. This is challenging as $N$ is not a stopping time for the filtration generated by the levels. We use techniques from rare-event simulation to overcome the challenge. Our strategy is to find the record-breakers sequentially until we find the last one. Let $\tau_k$ to denote the level of the $k$-th record break, i.e.
\$
\tau_0&=0, \\
\tau_k&=\inf \left\{n\ge \tau_{k-1}+1: \| B^H_{n}-B^H_{n-1}\|_{\infty}>\rho\cdot 2^{-(H-\delta)n}\right\}, \ k\ge1.
\$
Then we have 
\$
N=\sup\{k\ge 1: \tau_k<\infty \}.
\$
The general idea is as follows.
Conditional on $B^H_{\tau_k}$, 
we first use a change-of-measure to generate $B^H_{\tau_{k+1}}$. 
We then apply an acceptance-rejection step using a properly defined likelihood ratio. If the path is accepted, we find the $\tau_{k+1}$ and $B^H_{\tau_k}$,
otherwise, we claim that $\tau_k$ is the level of the last record-breaker, i.e. $N=\tau_k$.
The details of the algorithmic developments are provided in Section \ref{sec:alg}.

\section{Midpoint displacement of fBM} \label{sec:mid}
In this section, we analyze the midpoint displacement construction of fBM. This provides the theoretical foundation for our algorithmic development. Specifically, we establish the validity of the infinite series expansion \eqref{eq:main}, and analyze its rate of convergence under both the uniform norm and the $\alpha$-H\"older norm.
The analysis also provides us a way to construct the record-breakers. 

Recall that at the augmented points $D_{k}/D_{k-1}= \{t^{k}_{2j+1}\}_{j=0,1,\cdots,2^{k-1}-1} $, we have
\$ B^H_k\bigl(t^{k}_{2j+1}\bigr)-B^H_{k-1}\bigl(t^{k}_{2j+1}\bigr)=B^H(t^{k}_{2j+1})-\frac{1}{2}\bigl( B^H(t^{k-1}_j)+B^H(t^{k-1}_{j+1})\bigr). \$ 
For convenience, we denote by 
\# \label{ab} a^k_j:= B^H(t^{k}_{2j+1}),\quad b^k_j:=\frac{1}{2} \bigl( B^H(t^{k-1}_j)+B^H(t^{k-1}_{j+1})\bigr).\# 
Then since $ B^H_{k}(t)-B^H_{k-1}(t) $ is linear over intervals $[t^{k}_{2j},t^{k}_{2j+1} ]$ and $[t^{k}_{2j+1},t^{k}_{2j+2}]$, we have 
 \$
 \|B^H_k-B^H_{k-1} \|_{\infty} = \max_{0 \le j \le 2^{k-1}-1}| a^k_j-b^k_j  |.
 \$

We first establish the convergence rate of $\|B^H_k-B^H_{k-1} \|_{\infty}$,
which lays the foundation of subsequent results. 
We define
$$\ell_k:= 2^{-(H-\delta)k}, \mbox{ for any fixed $\delta\in (0,H)$.} $$
For any constant $\nu>0$, we denote by 
\# \label{start1}
K(\nu)=\sup\{n\ge1: 4\sqrt{n}>\nu\cdot 2^{\delta n}  \}.
\#
Then we have the following theorem establishing bounds for $\|B^H_k-B^H_{k-1} \|_{\infty}$. 
\begin{theorem} \label{tail}
{ For any constant $\nu, \nu^*>0$, for all $k> K(\nu)$, we have } 
\$
\PP\bigl(  \|B^H_k-B^H_{k-1} \|_{\infty}\ge \rho \ell_k \bigr)=\PP\Bigl(\max_{0\le j \le 2^{k-1}-1} |a^k_j-b^k_j | \ge \rho \ell_k\Bigr)  \le 2\exp \bigl\{-{\nu^*}^2 \cdot 2^{2k \delta-2}\bigr\},
\$
where $\rho=2(\nu+\nu^*)$.
\end{theorem}

Before we prove Theorem \ref{tail}, we would like to comment that $\nu$ and $\rho$ are parameters characterizing the record-breakers.
We have some freedom in choosing these parameters, and there is a tradeoff involved.
For larger $\rho$, the record-breakers are less likely to happen and it is relatively faster to find the last record-breaker. However, as a cost, we need to truncate at a higher level to achieve the desired accuracy. We provide more discussion about the choice of these parameters in practice in Section \ref{sec:num}.

\begin{proof}
Note that for fBM with Hurst index $H\ne 1/2$,  $a^k_j-b^k_j $ is not a centered Gaussian random variable. In the following, we use 
\$ 
c^k_j=\EE\bigl[ B^H\bigl(t^{k}_{2j+1}\bigr)\big |\bB^H_{k-1} \bigr],\ j=0,1,\cdots,2^{k-1}-1,
\$     
to denote the conditional expectation of fBM at the augmented points $D_{k}/D_{k-1} $ given the values of fBM on $D_{k-1}$.
Then we have 
\#  \label{decomp}
\max_{0\le j \le 2^{k-1}-1} |a^k_j-b^k_j | \le \max_{0\le j \le 2^{k-1}-1} |a^k_j-c^k_j |+ \max_{0\le j \le 2^{k-1}-1} |c^k_j-b^k_j |.
\#
The two terms in the right-hand side of inequality \eqref{decomp} correspond to the variance and bias. In what follows, we will establish bounds for each of them. 

It is easy to see that 
\begin{equation} \label{eq:decomp}
\PP\Bigl(\max_{0\le j \le 2^{k-1}-1} |a^k_j-b^k_j | \ge \rho\ell_k\Bigr)  \le \underbrace{\PP\Bigl(\max_{0\le j \le 2^{k-1}-1} |a^k_j-c^k_j | > \rho \ell_k/2\Bigr)}_{(V)}  + \underbrace{\PP\Bigl(\max_{0\le j \le 2^{k-1}-1} |c^k_j-b^k_j | > \rho \ell_k/2\Bigr)}_{(B)}.
\end{equation}

{ The rest of the proof is divided into two parts. We first establish a bound for (V), which corresponds to the variance. We then establish a bound for (B), which corresponds to the bias.
In subsequent analysis, we need several auxiliary results that are summarized in lemmas.}

For (V), we have 
\# \label{propdecomp}
\PP\Bigl(\max_{0\le j \le 2^{k-1}-1} |a^k_j-c^k_j | > \rho \ell_k/2\Bigr)=\E\Bigl[ \PP\Bigl(\max_{0\le j \le 2^{k-1}-1} |a^k_j-c^k_j | > \rho\ell_k/2\ |\ \bB^H_{k-1}\Bigr)\Bigr ].
\#  
In the following, we use $\PP_{k-1}(\cdot)$ to denote $\PP(\cdot|\bB^H_{k-1})$, which is the conditional probability given the values of $\bB^H_{k-1}$. We also use $ \EE_{k-1}  $ and $\VV_{k-1}  $ to denote corresponding conditional expectation and variance.  Then under the probability measure $\PP_{k-1}(\cdot)$,  $a^k_j$ is a Gaussian random variable with mean $  c^k_j$ and variance $\sigma^{2}_{kj}:=\VV_{k-1}(a^k_j) $. The following lemma upper bounds $\sigma^{2}_{kj}$ uniformly for all $j$'s.
\begin{lemma}\label{var_bound}
For all $k\ge1$ and $j=0,1,\cdots,2^{k-1}-1$, we have \$\VV_{k-1}(a^k_j)=\sigma^{2}_{kj} \le 2\cdot 2^{-2kH}.  \$  
\end{lemma} 
Using Lemma \ref{var_bound} and Borell-TIS inequality, for any $u>0$, we obtain 
\#  \label{b_tis}
\PP_k\biggl(\max_{0\le j \le 2^{k-1}-1} |a^k_j-c^k_j |-\EE_k\bigl[\max_{0\le j \le 2^{k-1}-1} |a^k_j-c^k_j |\bigr] > u \biggr) \le \exp \bigl\{-{u}^2 \cdot 2^{-2kH-2}\bigr\}.\# 
In order to get rid of the expectation in inequality \eqref{b_tis}, we need the following lemma to upper bound the expectation. 
\begin{lemma} \label{exp_bound}  
Let $X_1,X_2,\cdots,X_n$ be a sequence of (not necessarily independent) centered Gaussian random variables whose variances are uniformly bounded by $\sigma^2$. Then we have
\$
\EE\bigl[\max_{1\le i \le n}|X_i|\bigr]\le 2\sqrt{2\log(2n)}\cdot \sigma.
\$
\end{lemma}
By applying Lemma \ref{var_bound} and \ref{exp_bound}, we obtain that for any $\nu>0$, there exists a $K(\nu)$ such that for all 
$k\ge K(\nu)$,    
\#  \label{cexp}
\EE_{k-1}\Bigl[\max_{0\le j \le 2^{k-1}-1} |a^k_j-c^k_j |\Bigr] &\le 4\sqrt{k}\cdot 2^{-kH} \le \nu \ell_k. 
\#     
For any constant $\nu^*>0$, by setting $u=\nu^*\ell_k$ in inequality \eqref{b_tis} and using inequality \eqref{cexp}, we obtain 
\$
\PP_{k-1}\Bigl(\max_{0\le j \le 2^{k-1}-1} |a^k_j-c^k_j | > (\nu+\nu^*)\ell_k\Bigr) \le \exp \bigl\{-{\nu^*}^2 \cdot 2^{2k \delta-2}\bigr\}.
\$
Then based on \eqref{propdecomp}, for the unconditional probability, we have 
\# \label{tail1}
\PP\Bigl(\max_{0\le j \le 2^{k-1}-1} |a^k_j-c^k_j | >(\nu+\nu^*) \ell_k\Bigr) \le \exp \bigl\{-{\nu^*}^2 \cdot 2^{2k \delta-2}\bigr\}.
\#

Now, we turn to (B). In contrast to the previous proof where we deal with the conditional probability $\PP_{k-1}(\cdot)$ first, in this part, we consider the unconditional probability $\PP(\cdot)$ directly. In the following, to simplify notations, let $\bb_{k}=(b^k_0,\cdots,b^k_{2^{k-1}-1}) $ and $\bc_k=(c^k_0,\cdots,c^k_{2^{k-1}-1}) $. Then by definition, we have
 \$
\bb_k=\bM_{k-1} \bB^H_{k-1},\  \text{where}\  \bM_{k-1}= \left[
 \begin{matrix}
   1/2 & 1/2 & 0 & \cdots & 0 \\
   0 & 1/2 & 1/2 & \cdots & 0 \\
   \cdots & \cdots & \cdots & \cdots & \cdots \\
   0 & 0 & 0 & \cdots & 1/2 \\
  \end{matrix}
  \right]_{2^{k-1} \times (2^{k-1}+1)}. 
\$ 
Based on the conditional distribution of multivariate Gaussian random vector and the covariance function of fBM, we have 
$\bc_k=\bN_{k-1} \bB^H_{k-1}$, where  $\bN_{k-1}=\bSigma_{12}^{(k-1)}\cdot [\bSigma_{22}^{(k-1)}]^{-1}.$ Here, $\bSigma_{22}^{(k-1)}  $ and  $\bSigma_{12}^{(k-1)}  $ take the form
\$
\bSigma_{22}^{(k-1)}=
\left[
 \begin{matrix}
   0 & 0 & 0  & \cdots & \cdots  \\
  0 & 2^{2H} & (2^{2H}+4^{2H}-2^{2H})/2  & \cdots & \cdots  \\
   0 & (2^{2H}+4^{2H}-2^{2H})/2 & 4^{2H}  & \cdots & \cdots  \\
   \cdots & \cdots & \cdots  & \cdots &       \cdots   \\
   \cdots & \cdots & \cdots  & \cdots &   \cdots  \\
  \end{matrix}
  \right]_{(2^{k-1}+1) \times (2^{k-1}+1)}
\$
which is a $(2^{k-1}+1)$  by $(2^{k-1}+1)$ matrix with $(i,j)$-th entry $ (|2i-2|^{2H}+ |2j-2|^{2H}-|2i-2j|^{2H})/2$;
\$
\bSigma_{12}^{(k-1)}=
\left[
 \begin{matrix}
   1^{2H} & (1^{2H}+2^{2H}-1^{2H})/2 & (1^{2H}+4^{2H}-3^{2H})/2  & \cdots & \cdots  \\
  3^{2H} & (3^{2H}+2^{2H}-1^{2H})/2 & (3^{2H}+4^{2H}-1^{2H})/2  & \cdots & \cdots \\
   5^{2H} & (5^{2H}+2^{2H}-3^{2H})/2 & (5^{2H}+4^{2H}-1^{2H})/2  & \cdots & \cdots \\
   \cdots & \cdots & \cdots  & \cdots &       \cdots   \\
   \cdots & \cdots & \cdots  & \cdots &   \cdots  \\
  \end{matrix}
  \right]_{2^{k-1}\times (2^{k-1}+1)}
\$
which is a $2^{k-1}$  by $(2^{k-1}+1)$ matrix with $(i,j)$-th entry $ (|2i-1|^{2H}+ |2j-2|^{2H}-|2i-2j+1|^{2H})/2$. We remark that here the inverse $[\cdot]^{-1}$ is interpreted as generalized inverse. 

Then we have $ \bc_k-\bb_k=(\bN_{k-1}-\bM_{k-1})\bB^H_{k-1}$ and its covariance matrix is given by
\$
 \bSigma^{(k)}=(\bN_{k-1}-\bM_{k-1}) \bSigma_{22}^{({k-1})}  (\bN_{k-1}-\bM_{k-1})^{\top} \cdot \Delta_{k}^{2H}.
 \$  The following lemma bounds the diagonal entries of $ \bSigma^{(k)} $, which correspond to the variances of random variables $c^k_j- b^k_j $.  
\begin{lemma} \label{var_bound2}
The diagonal entries of $\bSigma^{(k)}$ are uniformly upper bounded by $2\cdot 2^{-2kH}$.
\end{lemma}
Then by using Borell-TIS inequality again, we have 
\$
\PP\biggl(\max_{0\le j \le 2^{k-1}-1} |c^k_j-b^k_j |-\EE\bigl[\max_{0\le j \le 2^{k-1}-1} |c^k_j-b^k_j |\bigr] > \nu^*\ell_k \biggr) \le \exp \bigl\{ -{\nu^*}^2 \cdot 2^{2k \delta-2} \bigr\}.
\$ 
Similar to the proof for (V), we can get rid of the expectation and obtain that for all $k>K(\nu)$,
\# \label{tail2}
\PP\Bigl(\max_{0\le j \le 2^{k-1}-1} |c^k_j-b^k_j | > (\nu+\nu^*)\ell_k\Bigr) \le \exp \bigl\{ -{\nu^*}^2 \cdot 2^{2k \delta-2} \bigr\}.
\#
Finally, combining \eqref{tail1} (for (V)) and \eqref{tail2} (for (B)), we obtain
\$
\PP\Bigl(\max_{0\le j \le 2^{k-1}-1} |a^k_j-b^k_j | \ge 2(\nu+\nu^*)\ell_k\Bigr)  \le 2\exp \bigl\{ -{\nu^*}^2 \cdot 2^{2k \delta-2} \bigr\},
\$ 
which concludes the proof of Theorem \ref{tail}. 
\end{proof}

\subsection{Validity of the expansion} \label{sec:valid}
Let $\cC([0,1])$ be the space of continuous functions over $[0,1]$ equipped with uniform norm $\| \cdot \|_{\infty}  $. 
The next theorem establish the validity of \eqref{eq:main}. 

\begin{theorem} \label{dyadic_con}
 The sample paths of $B^H_n(t) $ converge to a fBM $B^H(t) $ in $\cC([0,1])$ almost surely. In other words, 
 \$
 \PP\Bigl( \lim_{n \to \infty }\| B^H_n-B^H\|_{\infty}=0 \Bigr)=1.
 \$
\end{theorem}
\begin{proof}
We first prove that the sequence $\{ B^H_n \}_{n\ge 1}$  is a Cauchy sequence in $\cC([0,1])$ almost surely. Then its limiting process exists almost surely due to the completeness of $\cC([0,1])$. Second, we show that the limiting process is a Gaussian process and has the same covariance structure as the fBM. 

Since the tail bound $2\exp \{ -{\nu^*}^2 \cdot 2^{2k \delta-2} \}$ established in Theorem \ref{tail} is summable, by Borel-Cantelli Lemma, we have 
\$
\PP\Bigl(\max_{0\le j \le 2^{k-1}-1} |a^k_j-b^k_j | \ge \rho\ell_k,\ \text{i.o.}\ \Bigr)=0.
\$ 
Hence, there exists a random variable $N$, which is finite almost surely, such that for all $k\ge N$, $\max_{0\le j \le 2^{k-1}-1} |a^k_j-b^k_j |\le \rho\ell_k  $. Then for arbitrary
 $\epsilon>0$, when $n,m$ large enough, we have 
\$
\|B^H_n-B^H_m\|_{\infty} &\le \sum_{k=n+1}^{m}  \|B^H_{k}-B^H_{k-1}\|_{\infty}\le \sum_{k=n+1}^{m}\max_{0\le j \le 2^{k-1}-1} |a^k_j-b^k_j | \le\sum_{k=n+1}^{m} \ell_k<\epsilon .
\$
Thus, by definition, $ \{B^H_n\}_{n\ge1} $ is a Cauchy sequence in $\cC([0,1])$ almost surely. Since $\cC([0,1])$ is complete, there exists a stochastic process $X(t) $ such that $\|B^H_n-X\|_{\infty} $  converge to $0$ almost surely. 

We next show that $\{X(t)\}_{t\ge 0}$ is indeed a fBM. Consider  an arbitrary finite collection of time points $ (t^*_1,\cdots,t^*_m)\in [0,1]$. For each $1\le  i \le m$, there exists a sequence of points $t^{n_i}_n\in D_n$ such that  $t^{n_i}_n\to t^*_i$. Note that 
\$
\bigl|X(t^*)-B_n^H(t^{n_i}_n)\bigr|&\le \bigl|X(t^*)-X(t^{n_i}_n)\bigr|+\bigl|X(t^{n_i}_n)-B_n^H(t^{n_i}_n)\bigr| \\
&\le \bigl|X(t^*)-X(t^{n_i}_n)\bigr|+\|X-B_n^H\|_{\infty},
\$
and $B^H_n(t^{i_n}_n)$ is Gaussian. It implies that $ (X(t^*_1),\cdots,X(t^*_m))$ is the strong limit of a sequence of Gaussian random vectors, and hence, is itself also Gaussian. Thus, $ X(t)$ is a Gaussian process. Moreover, by the construction of $ B_n^H(t)$, the covariance matrix of $ (B_n^H(t^{n_1}_n),\cdots,B_n^H(t^{n_m}_n)) $ is 
$\bSigma^H(t^{n_1}_n,\cdots,t^{n_m}_n)$, where $ \bSigma^H$ is the covariance matrix function of fBM. Since $ \bSigma^H$ is continuous, we have $\bSigma^H(t^{n_1}_n,\cdots,t^{n_m}_n) \to \bSigma^H(t^{*}_1,\cdots,t^{*}_m)$, which implies that $ X $ has same covariance matrix function as fBM. Thus $X$ is  a fBM.  
\end{proof}

Theorem \ref{dyadic_con} indicates that the representation
\#  \label{rep}
 B^H(t)-B^H_n(t)=\sum_{k=n+1}^{\infty} \bigl[ B^H_k(t)-B^H_{k-1}(t)\bigr]
\#  
 is well defined.

\subsection{Convergence analysis in uniform norm and $\alpha$-H\"older norm} \label{sec:rate}
In this section, we study the convergence rate of \eqref{eq:main}. We first investigate the rate of convergence in uniform norm, which provides the basis for the $\epsilon$-strong simulation of fBM. Then, for fBM with Hurst index $H>1/2$, we strengthen our result under $\alpha$-H\"older norm, which is necessary for the $\epsilon$-strong simulation of fBM driven SDEs.    

Recall our definition of the record-breaker and the last breaking time $N$. 
Based on our analysis in Section \ref{sec:valid},
we say that a record breaker happens at level $k$ if 
\# \label{rb}  
\max_{0\le j \le 2^{k-1}-1} |a^k_j-b^k_j | \ge \rho \ell_k=\rho2^{-(H-\delta)k},
\# 
and
$
N=\sup \{k \ge 1:  \max_{0\le j \le 2^{k-1}-1} |a^k_j-b^k_j | \ge \rho \ell_k \}.
$ 
The following theorem shows that conditions C1) and C2) are satisfied for our definition of the record-breaker.
\begin{theorem} \label{ucr}
{For any fixed $\delta\in(0,H)$ and $t>0$, $ \EE[\exp\{tN\}]<\infty$. When $n>N$, 
$$\| B^H-B_n^H\|_{\infty}\leq\frac{\rho \cdot 2^{-(H-\delta)(n+1)} }{(1-2^{-(H-\delta)})}.$$  
}
\end{theorem}

\begin{proof}
The moment generating function of $N$ can be written as
\$
 \EE[\exp\{tN\}] =\EE\Bigl[ \int_0^{\infty} 1\bigl\{ \exp\{tN\}\ge u\bigr\} \ud u\Bigr]=\int_0^{\infty}  \PP(N\ge \log(u)/t) \ud u. 
\$
We have 
\$
\PP(N\ge \log(u)/t) &\le \sum_{k=[\log(u)/t]}^{\infty} \PP( \text{record broken at level $n$})\\
&\le \sum_{k=[\log(u)/t]}^{\infty}  2\exp \bigl\{-{\nu^*}^2 \cdot 2^{2k \delta-2}\bigr\} \mbox{ (by Theorem \ref{tail})} \\
&\le C \exp\bigl \{-{\nu^*}^2\cdot u^{2\delta/t}  \bigr\}, 
\$ 
where $C$ is a constant sufficiently large. Since $ \exp\{-{\nu^*}^2\cdot u^{2\delta/t}  \}$ is integrable, $\EE[\exp\{tN\}]$ is finite. 

Now for $n>N$,  according to the representation \eqref{rep}, we have     
\$
\bigl \|B^H-B_n^H\bigr \|_{\infty}  \le \rho \cdot \sum_{k=n+1}^{\infty} 2^{-(H-\delta)k}= \frac{\rho \cdot 2^{-(H-\delta)(n+1)} }{1-2^{-(H-\delta)}}.
\$
\end{proof}

For $0<\alpha<1$, the $\alpha$-H\"older norm of a function $f(\cdot)$ over interval $[0,1]$ is defined as 
\$
\|f  \|_{\alpha}=\sup_{0\le s<t \le1}\frac{|f(s)-f(t)|}{|s-t|^{\alpha}}.
\$
For the $\alpha$-H\"older norm, we only consider fBM $ B^H $ with $H>1/2$. Then for all $\alpha \in (1/2,H)$, the sample paths of $B^H $ are $\alpha$-H\"older continuous almost surely. By the representation \eqref{rep}, 
we have the following upper bound for the $\alpha$-H\"older norm of discretization error 
\$ 
\bigl\|B^H_n-B^H\bigr\|_{\alpha}\le\sum_{k=n+1}^{\infty}  \bigl\|B^H_k-B^H_{k-1}\bigr\|_{\alpha}. 
\$  
For each discretization level $k$, the following lemma gives a computable bound of $\|B^H_k(t)-B^H_{k-1}(t)\|_{\alpha}$. 
\begin{lemma} \label{holer_bound}
For all $k\ge1$, we have 
\$
 \bigl\|B^H_k-B^H_{k-1}\bigr\|_{\alpha}\le2^{\alpha(k-1)+2} \cdot \max_{0\le j \le 2^{k-1}-1} |a^k_j-b^k_j|.
\$  
\end{lemma}
The following theorem establishes convergence rate of \eqref{eq:main}  in the $\alpha$-H\"older norm. 
\begin{theorem} \label{hcr}
{ For any fixed $\alpha \in (1/2,H)$ and  $\delta\in(0,H-\alpha)$, when $n>N$}
$$\|B^H-B_n^H\|_{\alpha} \leq \frac{\rho2^{2-\alpha}\cdot 2^{-(H-\alpha-\delta)(n+1)}}{1-2^{-(H-\alpha-\delta)}}.$$
\end{theorem}
\begin{proof}
By the definition of $N$, for all $k>N$, we have  $\max_{0\le j \le 2^{k-1}-1} |a^k_j-b^k_j| \le \rho 2^{-(H-\delta)k} $. Then according to Lemma \ref{holer_bound}, we have 
\$
 \bigl\|B^H_k-B^H_{k-1}\bigr\|_{\alpha}\le2^{\alpha(k-1)+2}  \cdot \rho 2^{-(H-\delta)k}=\rho2^{2-\alpha}\cdot 2^{-(H-\alpha-\delta)k}.
\$
\end{proof}
   
As a result of Theorem \ref{hcr}, once we find $N$, we also have an upper bound for the $\alpha$-H\"older norm of the fBM sample path. Specifically,
\$\|B^H\|_{\alpha} \le \|B_N^H\|_{\alpha}+\frac{\rho2^{2-\alpha}\cdot 2^{-(H-\alpha-\delta)(N+1)}}{1-2^{-(H-\alpha-\delta)}}. \$

\section{Simulation Algorithm} \label{sec:alg}
In this section, we introduce our $\epsilon$-strong simulation algorithm in details. Based on the theoretical foundation built in the previous section, our simulation algorithm includes two main steps.  First, we simulate the fBM 
up to level $N$, where $N$ is the level of the last record-breaker.
Notice that once we find $N$, the truncation error at level $n>N$, is controlled by
$$ \rho \cdot \sum_{k=n+1}^{\infty}2^{-(H-\delta)k}= \frac{\rho \cdot 2^{-(H-\delta)(n+1)} }{1-2^{-(H-\delta)}}.$$
Second, we find the truncation level $N(\epsilon)$ such that
$(1-2^{-(H-\delta)})^{-1} \cdot \rho \cdot 2^{-(H-\delta)N(\epsilon)}\leq\epsilon.$
In this step, if $N(\epsilon) \leq N$, we have already obtained an  $\epsilon$-strong approximated sample path of fBM by simulating the path up to level $N$. 
Otherwise, we need to refine the path from level $N$ to level $N(\epsilon)$. 
We summarize our main simulation algorithm in Algorithm \ref{alg:main}. 
\begin{algorithm} 
\caption{$\epsilon$-Strong Simulation of fBM (SFBM)} \label{alg:main} 
\begin{algorithmic}[1]
\STATE \textbf{Input:} Hurst index $H$, simulation accuracy $\epsilon$, record-breaker parameter $\rho,\delta$.
\STATE \textbf{Find the last record-breaker:} 
\STATE \qquad Call Algorithm \ref{SLRB} (SLRB): set $[N,\text{SP}]\leftarrow $ SLRB($H,\rho,\delta$).
\STATE \textbf{Find the truncation level:} set $N(\epsilon) \leftarrow \max\bigl\{N,  \lceil \log_2( \rho\cdot (\epsilon (1-2^{-(H-\delta)}))^{-1} )/ (H-\delta)   \rceil    \bigr\}$.
 \STATE \textbf{If $N(\epsilon)>N$:}
\STATE  \qquad \textbf{Simulate until the level $N(\epsilon)$  using acceptance-rejection method:} 
\STATE  \qquad \qquad \textbf{Repeat:} sample fBM at $D_{N(\epsilon)}/D_{N}$, under the nominal measure conditional on the value of the fBM at $D_{N}$.
\STATE  \qquad \qquad \textbf{Until:} no record-breakers happen at levels $N+1,\cdots,N(\epsilon)$. 
\STATE  \qquad \qquad  Set SP $\leftarrow$ Union(SP, valus of fBM at $D_{N(\epsilon)}/D_{N}$).
\STATE \textbf{Output:} $B^H_{N(\epsilon)}(t)$, the piecewise linear interpolation of SP.
\end{algorithmic}
\end{algorithm}
The details of the first step (finding the last record-breaker) is further outlined in Algorithm \ref{SLRB}.
The second step (refining the dyadic approximation up to the truncation level) involves simple acceptance-rejection method which is already detailed in Algorithm \ref{alg:main}. 
We note that sampling the fBM at $D_{N(\epsilon)}$ given its values at $D_N$ can be implemented straightforwardly by Cholesky decomposition given the conditional mean and covariance matrix. 
However, this implementation has high computational cost.
We will provide a more efficient recursive implementation in Section \ref{cc}.

In Algorithm \ref{SLRB}, we find the record-breakers sequentially until the last one. 
Finding the next record-breaker is a challenging task, as the next record-breaker may never happen.
We overcome the difficulty here by using techniques from rare-event simulation. 
Intuitively, breaking the record at level $n$ is a rare event for large values of $n$.
The details of how to find the next record-breaker are summarized in Algorithm \ref{SNRB}.

\begin{algorithm}
\caption{Simulation of the Last Record-Breaker (SLRB)} \label{SLRB}
\begin{algorithmic}[1]
\STATE \textbf{Input:} Hurst index $H$, record-breaker parameter $ \rho,\delta$.
\STATE \textbf{Determine the starting level:}  
\STATE  \qquad   \qquad  \quad Calculate $N=N^*(\rho,\delta)$. Sample $\bB^H_{N}$ and store it in array SP.
\STATE  \qquad   \qquad  \quad  Set $I\leftarrow 1$.
\STATE \textbf{While $I=1$:}  
\STATE \qquad \textbf{Find the next record-breaker:} 
\STATE \qquad \qquad Call Algorithm \ref{SNRB} (SNRB): set $ [I, N,\text{SP}]\leftarrow $SNRB$(N, \bB^H_{N},H,\rho,\delta)$
\STATE \textbf{Output:} $[N,\text{SP}]$.
\end{algorithmic} 
\end{algorithm}
\begin{remark}
In Algorithm \ref{SLRB}, we need to start the dyadic approximation from a nontrivial stating level $N^*(\rho,\delta)$. 
$N^*(\rho,\delta)$ is defined later in equation \eqref{start2} for technical reasons. 
Specifically, $N^*(\rho,\delta)$ is to ensure a proper bound on the likelihood ratio of the change-of-measure that we will apply in Algorithm \ref{SNRB}. 
In practice, we can twist the record-breaker parameters, $(\rho,\delta, \nu, \nu^*)$, to obtain a reasonable starting level. 
\end{remark}

In Algorithm \ref{SNRB}, we apply a change of measure technique to find the next record-breaker or claim that the record will never be broken again. The basic idea is as follows.
Assuming that we have already simulated the fBM until level $n$. We denote by $\tau$ the level of the next record-breaker.
Our goal is to find the next record-breaker or claim that $\tau=\infty$, which means the record will never be broken again. In order to determine whether $\tau<\infty$, we essentially want to generate a Bernoulli random variable with success probability 
$\PP_{n}(\tau<\infty):=\PP(\tau<\infty|\bB_{n}^H).$ 
However, the exact value of $\PP_{n}(\tau<\infty)$ is intractable. To overcome this difficulty, note that we can rewrite the probability as
\begin{equation}\label{eq:ber_main}
\PP_{n}(\tau<\infty)=\sum_{m=1}^{\infty}\PP_{n}(\tau=n+m)=\sum_{m=1}^{\infty}\frac{\PP_{n}(\tau=n+m)}{g_{n}(m)}\cdot g_{n}(m),
\end{equation}
where $\{ g_{n}(m)\}_{m\ge1} $ is a carefully designed distribution taking value in $\ZZ^+$ such that 
$${\PP_{n}(\tau=n+m)}/ {g_{n}(m)}\le 1,$$ 
for all $m\ge 1$. Note that $\{ g_{n}(m)\}_{m\ge1} $ can be interpreted as a potential realization of $\tau-n$.

We now introduce our specific choice of $\{ g_{n}(m)\}_{m\ge 1}$.
\#  \label{gm}
g_{n}(m)=Z_n^{-1}\cdot 2^{{n}+m}\cdot \exp\bigl\{-\rho^2/8\cdot 2^{2({n}+m)\delta} \bigr \}, 
\#
where $Z_n$ is the normalizing constant such that the sum of $ g_{n}(m)$  equals to one.
If we first generate $M$ from $\{ g_{n}(m)\}_{m\ge1} $ and then, based on the value of $M$, generate a Bernoulli random variable with success probability ${\PP_{n}(\tau=n+M)}/ {g_{n}(M)} $. Then it is easy to see from equation \eqref{eq:ber_main} that we obtain a Bernoulli random variable whose success probability is $\PP_{n}(\tau<\infty)$. Moreover, if the Bernoulli trail is a success, we also know that the next record-breaker will happen at level $M$.


We next provide the general idea as how to generate a Bernoulli random variable with success probability ${\PP_{n}(\tau=n+M)}/ {g_{n}(M)}$.
Given a realization $M=m$, in order to generate the desired Bernoulli random variable, we apply the change-of-measure technique. Note that $\PP_n(\tau=n+m)=\EE_{n}[1\{\tau=n+m \}]$. 
Thus, we first sample $\tau$ and the associated fBM sample paths $\omega$ from a properly constructed probability measure $\QQ^{(m)}_{n}$. Then, we have 
\# \label{intuition}
\EE_{\QQ^{(m)}_{n}}\left[ \frac{\ud\PP_{n}(\omega)}{\ud\QQ^{(m)}_{n}(\omega)}\cdot \frac{ 1\{\tau=n+m\}}{g_{n}(m)}\right]=\frac{\PP_{{n}}(\tau=n+m)}{g_{n}(m)},
\# 
where $\EE_{\QQ^{(m)}_{n}}$ is the expectation with respect to $\QQ^{(m)}_{n}$. If we can upper bound the likelihood ratio $\ud\PP_{n}/\ud\QQ^{(m)}_{n}$ by $g_n(m)$, then we generate $U$, a uniform random variable over $[0,1]$ independent of everything else.  When
\$U< \frac{\ud \PP_{n}(\omega)}{\ud \QQ^{(m)}_{n}(\omega)}\cdot \frac{1\{\tau=n+m\}}{g_{n}(m)},\$ we accept $\omega$ as the trajectory leading to the next record-breaker, i.e. we get $\bB_{\tau}^H$; otherwise, we claim that $\tau=\infty$.

The actual construction of change-of-measure involves more subtleties. For example, we not only need to consider the level of the next record-breaker, $n+m$, but also the actual time index, and whether we break the record due to a large positive deviation or negative deviation. The details of these subtleties are deferred to Section \ref{sec:ECM}. It is also in general not easy to bound the likelihood ratio $\ud\PP_{n}/\ud\QQ^{(m)}_{n}$. We thus introduce another technical step before we give the actual algorithm.
We define the \textquotedblleft bounded conditional expectation\textquotedblright  condition (BCE). We will later show in Section \ref{sec:ECM} that under this condition, the likelihood ratio is properly bounded. In addition, this condition only gets violated a finite number of times almost surely. 
\begin{definition} \label{BCE} (BCE condition) We say that $\bB^H_n$ satisfies the bounded conditional expectation condition, if for all $m\ge1$ and $1\le k \le 2^{n+m-1}$, 
\# \label{bce}
|\mu_n(m,k)|=\bigl | \bbeta^{\top} \EE[\balpha_{n}(m,k) | \bB^H_n] \bigr|  \le \rho/2\cdot \ell_{n+m},
\# 
where 
\#   \label{bal}
\bbeta=(1/2,-1,1/2)^{\top},\ 
\balpha_{n}(m,k)=(B^H(t^{{n}+m}_{2k-2}),B^H(t^{{n}+m}_{2k-1}),B^H(t^{{n}+m}_{2k}))^{\top}.
\#  
\end{definition}
When simulating the next record-breaker, we will first check if $ \bB^H_n$ satisfies the BCE condition. If not, we will keep generating more refined levels under the nominal measure until the BCE condition is satisfied, before we apply the change-of-measure. 
The details of how to check whether the BCE condition is met are laid out in Algorithm \ref{BCEC}, which we defer to Section \ref{sec:ECM} after we introduce a few more technical results.
Similarly, we also defer the details of the change-of-measure to Algorithm \ref{BCEC} in Section \ref{sec:ECM}.

\begin{algorithm}
\caption{Simulation of the Next Record-Breaker (SNRB)} \label{SNRB}
\begin{algorithmic}[1]
\STATE \textbf{Input:}  Current level $n$ and values $\bB^H_{n}$,  Hurst index $H$, record-breaker parameter $\rho,\delta$. 
\STATE \textbf{Initialize:} Set AP, an array of values of fBM at augmented points, to be NULL, and $J\leftarrow 0$. 
\STATE \textbf{Refine dyadic approximation until BCE condition is satisfied:}
\STATE \qquad \textbf{While $J=0$: } 
\STATE \qquad \qquad \textbf{Checking BCE condition:} 
\STATE \qquad \qquad \qquad Call Algorithm \ref{BCEC} (BCEC): set $J\leftarrow \text{BCEC}(H,\rho,\delta,n,\bB^H_{n})$.
\STATE \qquad \qquad \qquad \textbf{If $J=1$: break.}
\STATE \qquad \qquad \textbf{Refine dyadic approximation to next level:} 
\STATE \qquad \qquad \qquad  Sample fBM at $D_{n+1}/D_{n}$, under the nominal measure conditional on $\bB^H_{n}$, and then store it in AP. Update $n\leftarrow n+1$, $\bB^H_{n}\leftarrow \text{Union}(\bB^H_{n},\text{AP})$.
\STATE \qquad \qquad \qquad  \textbf{If a record-breaker happens at level $n$: break.} 
\STATE \textbf{If $J=1$, apply change-of-measure:} 
\STATE \qquad Sample $M$ from distribution $ \{g_{n}(m)\}_{m\ge 1}$. 
\STATE \qquad  Call Algorithm \ref{ECM} (ECM): set $[I,\text{AP}]\leftarrow $ECM($H,\rho, \delta,n,\bB^H_{n},M$).\\ 
\STATE \textbf{Output:} 
\STATE \qquad \textbf{If $J=0$ :} return $[1, n, \bB^H_{n}] $,
\STATE \qquad \textbf{Else if $I=1$ :} return $[1, n+M, \text{Union}(\bB^H_{n},\text{AP})] $,
\STATE \qquad  \textbf{Else:} return: $[0, n ,\bB^H_{n}] $.    
\end{algorithmic}
\end{algorithm}

\subsection{Change of measure}\label{sec:ECM}

In this section, we provide details of our construction of a new measure under which the record-breaker is more likely to happen.
Recall that the setting is that we have already generated $\bB^H_{n}$ and a proposed next record breaking level $n+m$, where $m$ is sampled from distribution $\{g_n(m)\}_{m\ge 1}$. Our goal is to find a way to generate a path such that the next record-breaker is more likely to happen at level $n+m$, and the likelihood ratio can be properly bounded.

Recall the definition of $\balpha_n(m,k)$ and $\bbeta$ in equation \eqref{bal}. Then based on the definition of the record-breakers, we say that a record is broken at level $n+m$, position $k$, if
\$
|\bbeta^{\top} \balpha_n(m,k) |>\rho \ell_{n+m}=\rho 2^{-(H-\delta)(n+m)}.
\$ 
Furthermore, we say that the record-breaker is up-crossing if
\$
\bbeta^{\top} \balpha_n(m,k) >\rho\ell_{n+m},
\$ 
and downward-crossing if
\$
\bbeta^{\top} \balpha_n(m,k) <-\rho \ell_{n+m}.
\$   
Let 
$$\Xi_{n}^{(m,k)}(\theta)=\EE_n[ \exp\{ \theta\cdot \bbeta^{\top} \balpha_n(m,k) \}]=\EE\bigl[ \exp\{ \theta\cdot \bbeta^{\top} \balpha_n(m,k) \}      | \bB^H_{n}\bigr], $$
which is the moment generating function of $\bbeta^{\top} \balpha_n(m,k)$ conditional on the value of $\bB^H_{n} $. 
We also assume that the conditional probability density of $\balpha_{n}(m,k) $ under measure $\PP_{n}(\cdot)$  is $ \psi_{n}^{(m,k)}.$ 
We also denote the conditional covariance matrix and expectation of $\balpha_{n}(m,k)$ by $\bSigma_{\balpha_n(m,k)}$ and $\bmu_n(m,k)$, respectively.

In what follows, we start by introducing the change-of-measure under the BCE condition. We then show that the BCE condition can  only be violated a finite number of times almost surely. Under the BCE condition, we first sample $K$ from the set $\{1,2,\cdots,2^{{n}+m-1}\}$ uniformly. The random variable $K$ roughly proposes the position of the next record-breaker. 
We also sample $\pi$ from the set $\{+,-\}$ uniformly. The random variable $\pi$ roughly proposes whether the record-breaker is up-crossing ($+$) or downward-crossing ($-$). Second, given $M=m, K=k$, if $\pi=+$, we apply exponential tilting to $\psi_{n}^{(m,k)}$ with tilting parameter  
\begin{equation}\label{eq:tilt}
\theta_{n}^+(m)=\rho/2\cdot 2^{(m+{n})(H+\delta)} \footnote{The tilting parameter $\theta_{n}^+(m)$ is carefully chose to make sure that the record-breaking event is more likely to happen under the tilted measure and the likelihood ratio is suitably bounded. See Section \ref{sec:alg_proof} for more details.}.
\end{equation}
Specifically, we sample $\balpha_{n}(m,k) $ from the density
\$ 
\tilde{\psi}_{n}^{(m,k,+)}(x_1,x_2,x_3) = \psi_{n}^{(m,k)}(x_1,x_2,x_3) \cdot \exp\Bigl\{\theta^+_{n}(m)\cdot \bigl(\frac{1}{2}(x_1+x_3)-x_2  \bigr)-\log\bigl(\Xi_{n}^{(m,k)}(\theta^+_{n}(m))\bigr) \Bigr\}, 
\$ 
Note that the tilted distribution $\tilde{\psi}_{n}^{(m,k,+)} $ is still Gaussian. 
In particular, 
$\tilde{\psi}_{n}^{(m,k,+)} $ is the density of the multivariate Gaussian
$$N\bigl( \bmu_{n}(m,k)+\theta^+_{n}(m)\cdot \bSigma_{\balpha_n(m,k)}\bbeta,\bSigma_{\balpha_n(m,k)}\bigr).$$
If  $\pi=-$, we apply the exponential tilting to $\psi_{n}^{(m,k)}$ with tilting parameter  $\theta_{n}^-(m)=-\theta_{n}^+(m)$.

After we have sampled $\balpha_{n}(m,k) $ under the tilted measure, given the values of $\balpha_{n}(m,k)  $ and $ \bB^H_{n} $, we sample fBM at the remained dyadic points \$D_{{n}+m}/ (D_{n}\cup\{t^{{n}+m}_{2k-2},t^{{n}+m}_{2k-1},t^{{n}+m}_{2k}\})\$ under the nominal measure. This step is achieved by calculating the conditional expectation and covariance matrix, and then sampling from the corresponding multivariate Gaussian distribution. We use $\QQ_n$ to denote the tilted measure introduced as above.
We also denote $\QQ^{(m,k,+)}_{n}$ as the conditional probability measure $\QQ_{n}(\cdot| M=m, K=k, \pi=+)$. Let 
\$
\Theta^+_{n}(m,k)&:={\ud \PP_{n}}/{\ud \QQ^{(m,k,+)}_{n}}\cdot  g_n(m)^{-1}\cdot 2^{n+m}\\
&=g_{n}(m)^{-1}\cdot 2^{{n}+m}\cdot \exp\Bigl\{ -\theta_{n}^+(m)\cdot \bbeta^{\top} \balpha_n(m,k)+\log\bigl(\Xi_{n}^{(m,k)}(\theta_{n}^+(m))\bigr)   \Bigr \}.\$ 
Similarly, we define 
$\QQ^{(m,k,-)}_{n}$ as the conditional measure $\QQ_{n}(\cdot| M=m, K=k, \pi=-)$ and  
\$
\Theta^-_{n}(m,k)&:={\ud \PP_{n}}/{\ud \QQ^{(m,k,+)}_{n}}\cdot  g_n(m)^{-1}\cdot 2^{n+m}\\
&=g_{n}(m)^{-1}\cdot 2^{{n}+m}\cdot \exp\Bigl\{ -\theta_{n}^-(m)\cdot \bbeta^{\top} \balpha_n(m,k)+\log\bigl(\Xi_{n}^{(m,k)}(\theta_{n}^-(m))\bigr)   \Bigr \}.\$ 
From the definition of $g_n(m)$ in \eqref{gm}, the normalizing constant $Z_n$ is given by 
$$
Z_n=\sum_{m=1}^{\infty} 2^{{n}+m}\cdot \exp \{-\rho^2/8\cdot 2^{2({n}+m)\delta}  \}.
$$
We denote by
\# \label{start2}
N^*(\rho,\delta)=1+\sup\{n\ge1: Z_n>  1\}.  
\#
Note that for all $n \ge N^*(\rho,\delta) $, we have $Z_n \le 1 $. The next lemma shows that under the BCE condition, $\Theta^{+}_{n}(m,k)$'s and $\Theta^{-}_{n}(m,k)$'s are suitably bounded.
This result is important in constructing our change-of-measure procedure. 
\begin{lemma} \label{lrb}
For $n\ge N^*(\rho,\delta)$, under the BCE condition, 
when $M=m$, $K=k$, and $\pi=+$,  
$$\Theta^+_{n}(m,k)\cdot 1\left\{(1/2,-1,1/2)^{\top} \balpha_{n}(m,k)>\rho \ell_{{n}+m} \right\}\leq 1;$$ 
when $M=m$, $K=k$, and $\pi=-$, 
$$\Theta^-_{n}(m,k)\cdot 1\left\{(1/2,-1,1/2)^{\top} \balpha_{n}(m,k)<-\rho \ell_{{n}+m} \right\}\leq 1.$$ 
\end{lemma}
We're now ready to present our actual algorithm.  We denote by 
\$
R_{{n}+m}=\sum_{k=1}^{2^{{n}+m-1}}1 \bigl\{ |\bbeta^{\top} \balpha_n(m,k)|>\rho \ell_{{n}+m}\bigr\},
\$
 the total number of record-breakers at level $n+m$. We also generate a uniformly distributed random variable $U$ over interval $[0,1]$, independent of everything else. When $\pi=+$, we return $1$, if 
\$
U< \Theta^+_{n}(m,k) \cdot  R_{{n}+m}^{-1}\cdot 1\left\{\bigl\{ \bbeta^{\top} \balpha_n(m,k) >\rho \ell_{{n}+m} \bigr\}  \cap \cap_{j=1}^{m-1}[\cC_{n}(j)]^c\right\}, 
\$
and return $0$ otherwise \footnote{We define $0/0=0$ by convention.}. Here $ C_{n}(j)$ denotes the event that there are record-breakers at level $n+j$. When $\pi=-$, the procedure is similar. The details of this simulation procedure are summarized in Algorithm \ref{ECM} and we prove that it works in Section \ref{sec:alg_proof}.
\begin{algorithm}
\caption{Exponential Change of Measure (ECM)} \label{ECM}
\begin{algorithmic}[1]
\STATE \textbf{Input:} Hurst index $H$, record-breaker parameters $\rho,\delta$, current level $n$ and values $\bB^H_{n}$, level difference $m$. 
\STATE \textbf{Initialize:} set AP, the array of values of fBM at augmented points $D_{n+m}/D_{n}$, to be NULL.
\STATE \textbf{Sample position of  record-breaker:}  sample $K=k$ uniformly from $\{1, 2, \cdots, 2^{n+m-1} \}. $
\STATE \textbf{Determine up-crossing or downward-crossing:}  sample $\pi$ uniformly from $\{+,- \}. $
\STATE \textbf{Sample candidate path:} 
\STATE \qquad  \textbf{If $\pi=+$:} sample $\balpha_{n}(m,k)  $ from the exponential tilted measure $\tilde{\psi}_{n}^{(m,k,+)}$.
\STATE \qquad  \textbf{If $\pi=-$:} sample $\balpha_{n}(m,k)  $ from the exponential tilted measure $\tilde{\psi}_{n}^{(m,k,-)}$.
\STATE \qquad  Sample fBM at the remained points until the discretization level $n+m$.
\STATE \textbf{Calculate likelihood ratio:} 
\STATE \qquad  \textbf{If $\pi=+$:} calculate $ \Theta_{n}^+(m,k) $. 
\STATE \qquad  \textbf{If $\pi=-$:} calculate $ \Theta_{n}^-(m,k) $. 
\STATE \textbf{Calculate total number of record-breakers at level $n+m$:} 
\STATE \quad \quad Set $R_{{n}+m}\leftarrow \sum_{k=1}^{2^{{n}+m-1}}1\bigl\{ | \bbeta^{\top} \balpha_n(m,k)|>\rho \ell_{{n}+m}\bigr\}. $
\STATE \textbf{Sample from uniform distribution:} generate $U\sim  \cU[0,1]$. 
\STATE \textbf{Determine the next record-breaker via acceptance-rejection:} 
\STATE \qquad  \textbf{If $\pi=+$:} 
set 
$$I \leftarrow 1\left\{U< \Theta^+_{n}(m,k) \cdot  R_{{n}+m}^{-1}\cdot 1\Bigl\{  \bigl\{ \bbeta^{\top} \balpha_n(m,k) >\rho \ell_{{n}+m}\bigr \} \cap \cap_{j=1}^{m-1}[\cC_{n}(j)]^c\Bigr\} \right\}.$$
\vspace{-5mm}
\STATE \qquad  \textbf{If $\pi=-$:} 
set
$$I\leftarrow 1\left\{U< \Theta^-_{n}(m,k) \cdot  R_{{n}+m}^{-1}\cdot 1\Bigl\{\bigl\{ \bbeta^{\top} \balpha_n(m,k) <-\rho \ell_{{n}+m} \bigr \} \cap \cap_{j=1}^{m-1}[\cC_{n}(j)]^c\Bigr\}\right\}.$$
\vspace{-5mm}
\STATE \textbf{If $I=1$:} AP=values of fBM at augmented points $D_{n+m}/D_{n}$.
\STATE \textbf{Output:} $I$ and AP.
\end{algorithmic}
\end{algorithm}

\begin{algorithm}
\caption{BCE Checking (BCEC)} \label{BCEC}
\begin{algorithmic}[1]
\STATE \textbf{Input:} Hurst index $H$, record-breaker parameters $\rho, \delta$, current level $n$ and values $\bB^H_{n}$.
\STATE \textbf{Initialize:} set $J\leftarrow 1$.
\STATE \textbf{Determine maximal checking level:}  
\STATE \qquad  Set $\gamma_n\leftarrow \max_{1\le i\le2^{n}+1} |(\bSigma_{n}^{-1}\bB^H_{n})_i|$. 
\STATE \qquad  Set $M_n \leftarrow \max \bigl\{1,  \lceil (H-\delta)^{-1}\cdot {\log_2\bigl((2^{n+1}+2)\gamma_n/\rho\bigr)}{}-n  \rceil \bigr\}.  $ 
\STATE  \textbf{For $1\le m \le M_n$: }
\STATE \qquad \textbf{For $1\le k \le 2^{n+m-1}$: }
\STATE \qquad \qquad  {Calculate conditional expectation $\mu_n(m,k)$. }
\STATE \qquad \qquad  \textbf{If $|\mu_n(m,k)|\ge \rho/2 \cdot 2^{-(n+m)(H-\delta)}$:} set $J \leftarrow 0$, break. 
\STATE \textbf{Output:} Return $J$. 
\end{algorithmic}
\end{algorithm}

It remains to show that the BCE condition is only violated a finite number times. In addition, we also need to 
have an efficient way to check whether the BCE condition is satisfied at a specific level $n$. 

Let \$ \cE_{n}=\{  | \mu_{n}(m,k) | > \rho/2 \cdot 2^{-(n+m)(H-\delta)},\ \text{for some}\ m\ge1\ \text{and}\ 1\le k \le 2^{n+m-1}\},\$
where $ \mu_{n}(m,k)$ is the conditional expectation defined in \eqref{bce}. By Definition \ref{BCE}, $\cE_{n}$ denotes the event that the BCE condition is violated at level $n$. We also define 
\$
\mathcal{N}_\cE=\sup \{ n\ge1: \cE_n\  \text{happens}\}
\$
\begin{lemma} \label{finitebce}
The BCE condition is violated a finite number times almost surely, i.e.
\$\PP\Bigl(\limsup_{n\rightarrow\infty} \cE_n\Bigr)=0.\$
Moreover, $\mathcal{N}_\cE$ has a finite moment generating function.
\end{lemma}  


We next introduce the algorithm to check whether the BCE condition is satisfied at level $n$ given the value of $\bB^H_n$.
The development of the algorithm is highly nontrivial. In particular, directly checking by definition is impractical, as we needs to evaluate $\mu_n(m,k)$ for infinitely many terms. To overcome this difficulty, we need to further explore the structure of $\mu_n(m,k)$. We use $\gamma_{n}$ to denote the maximal absolute value of the entries in vector $ \bSigma_{n}^{-1}\bB^H_{n} $, where $\bSigma_{n}$ is the covariance matrix of  $\bB^H_{n} $. We have the following lemma to characterize the bound for $\mu_n(m,k)$.
\begin{lemma} \label{ceb}
For all $m\ge1$ and $1\le k \le 2^{n+m-1}$, we have 
\$
|\mu_{n}(m,k)| \le \gamma_{n}\cdot (2^{n}+1)\cdot 2^{-2({n}+m)H}.
\$
\end{lemma}

According to Lemma \ref{ceb}, for a fixed level $n$ and values $\bB^H_n$, the decay rate of $ |\mu_{n}(m,k)|$ (with respect to $m$) is $O(2^{-2mH})$. However, to check the BCE condition, we only need to compare $ |\mu_{n}(m,k)|$ with $\rho/2 \cdot 2^{-(n+m)(H-\delta)}$, whose decay rate (with respect to $m$) is $O(2^{-m(H-\delta)})$. Hence, when 
\$
m \ge M_n=\max \Bigl\{1, \big\lceil (H-\delta)^{-1}\cdot {\log_2\bigl((2^{n+1}+2)\gamma_n/\rho\bigr)}{}-n \big\rceil \Bigr\},
\$
the following inequality 
\$
|\mu_{n}(m,k)| \le \rho/2 \cdot 2^{-(n+m)(H-\delta)}
\$
always holds. This implies that to check whether BCE condition holds at level $n$, we only need to calculate a finite collection of $\mu_{n}(m,k)$, i.e. $m=1, \dots, M_n$. The details is summarized in Algorithm \ref{BCEC}, which outputs $1$ when the BCE condition is satisfied and $0$, otherwise.

\subsection{Proof of algorithm} \label{sec:alg_proof}
In this section, we provide detailed proof to show that our algorithm actually works. The proof is divided into two main steps (two Theorems).
We first show that under the BCE condition, for the output of Algorithm \ref{SNRB}, $\QQ_n(M=m, I=1)=\QQ_{n}(\tau=n+m)$ and $\P(I=0)=\P_{n}(\tau=\infty)$ (Theorem \ref{thmA1}),  where $\QQ_n$ denotes the measure induced by Algorithm \ref{SNRB}.
We then show that when $M=m$ and $I=1$, the output path, i.e. the fBM on the augmented points $D_{{n}+m}/D_{{n}}$,
follows from the distribution $\PP_{n}(\cdot|\tau=n+m)$ (Theorem \ref{thmA2}). 

 
\begin{theorem} \label{thmA1}
For the output of Algorithm \ref{SNRB}, when $n\ge N^*(\rho,\delta)$ and the BCE condition holds, 
$I$, is a Bernoulli random variable with success probability $\P_{n}(\tau<\infty)$. 
Moreover, \$\QQ_n(M=m, I=1)=\P_{n}(\tau=n+m).\$ 
\end{theorem}

\begin{proof}
By definition, we have 
\$
\QQ_{n}(I=1)=\sum_{m=1}^{\infty}\QQ_{n}(I=1,M=m)=\sum_{m=1}^{\infty}\QQ_{n}(I=1|M=m)\cdot g_n(m).
\$ 
We next show that $\QQ_{n}(I=1|M=m)=\PP_{n}(\tau=n+m)/ g_{n}(m)$. Note that if $X$ is a random variable taking value in $[0,1]$ and $U$ is a uniformly distributed random variable on $[0,1]$ independent of $X$, then $1\{U<X\}$ is a Bernoulli random variable of success probability $\EE[X]$. Recall that in our algorithm, both $\Theta^+_{n}(m,k)$ and $\Theta^+_{n}(m,k)$ are bounded by one. Then by the definition of $I$ and Lemma \ref{lrb}, we have the following decomposition
\$
& \QQ_{n}(I=1|M=m)\\
&=\EE_{\QQ^{(m)}_{n}}\Bigl[1\left\{\{  \bbeta^{\top} \balpha_n(m,K) > \ell_{{n}+m} \} \cap \cap_{j=1}^{m-1}[\cC_{n}(j)]^c\right\}\cdot \Theta^+_{n}(m,K) \cdot  R_{{n}+m}^{-1}\Big|\pi=+\Bigr] \cdot\frac{1}{2}\\
&\qquad +\EE_{\QQ^{(m)}_{n}}\Bigl[1\left\{\{  \bbeta^{\top} \balpha_n(m,K) <- \ell_{{n}+m} \} \cap \cap_{j=1}^{m-1}[\cC_{n}(j)]^c\right\}\cdot \Theta^-_{n}(m,K) \cdot  R_{{n}+m}^{-1}\Big|\pi=-\Bigr] \cdot\frac{1}{2}\\
&=\sum_{k=1}^{2^{{n}+m-1}}\EE_{\QQ^{(m,k,+)}_{n}}\Bigl[1\bigl\{\{   \bbeta^{\top} \balpha_n(m,k) > \ell_{{n}+m} \} \cap \cap_{j=1}^{m-1}[\cC_{n}(j)]^c\bigr\}\cdot \Theta^+_{n}(m,k) \cdot  R_{{n}+m}^{-1}\Bigr]\cdot \frac{1}{2^{n+m}}\\
&\qquad +\sum_{k=1}^{2^{{n}+m-1}}\EE_{\QQ^{(m,k,-)}_{n}}\Bigl[1\bigl\{\{   \bbeta^{\top} \balpha_n(m,k) > \ell_{{n}+m} \} \cap \cap_{j=1}^{m-1}[\cC_{n}(j)]^c\bigr\}\cdot \Theta^+_{n}(m,k) \cdot  R_{{n}+m}^{-1}\Bigr]\cdot \frac{1}{2^{n+m}}.
\$
By the definition of weighted likelihood ratio $\Theta^+_{n}(m,k)  $ and $ \Theta^-_{n}(m,k) $, we further have 
\$
& \QQ_{n}(I=1|M=m)\\
&=\sum_{k=1}^{2^{{n}+m-1}}\EE_{\QQ^{(m,k,+)}_{n}}\left[1\left\{\{  \bbeta^{\top} \balpha_n(m,k) > \ell_{{n}+m} \} \cap \cap_{j=1}^{m-1}[\cC_{n}(j)]^c\right\} \cdot  R_{{n}+m}^{-1}\cdot\frac{\ud{\PP}_{n}}{\ud\QQ^{(m,k,+)}_{n}} \cdot g_{n}(m)^{-1}\right]\\
&\qquad +\sum_{k=1}^{2^{{n}+m-1}}\EE_{\QQ^{(m,k,-)}_{n}}\left[1\left\{\{  \bbeta^{\top} \balpha_n(m,k) < -\ell_{{n}+m} \} \cap \cap_{j=1}^{m-1}[\cC_{n}(j)]^c\right\} \cdot  R_{{n}+m}^{-1}\cdot\frac{\ud{\PP}_{n}}{\ud\QQ^{(m,k,-)}_{n}}\cdot g_{n}(m)^{-1} \right].
\$
Note that as
\$
  1\left\{ | \bbeta^{\top} \balpha_n(m,k)| > \ell_{{n}+m}  \right\}&=1\left\{\bbeta^{\top} \balpha_n(m,k) > \ell_{{n}+m}\right\} +1\left\{  \bbeta^{\top} \balpha_n(m,k) <- \ell_{{n}+m} \right\},
\$
we finally have  
\$
\QQ_{n}(I=1|M=m)&=\sum_{k=1}^{2^{{n}+m-1}}\EE_{{n}}\Bigl[ 1\left\{\{  |\bbeta^{\top} \balpha_n(m,k) | > \ell_{{n}+m} \} \cap \cap_{j=1}^{m-1}[\cC_{n}(j)]^c\right\} \cdot  R_{{n}+m}^{-1} \Bigr]\cdot g_{n}(m)^{-1}\\
&=\EE_{{n}}\left[ 1\left\{\cC_{n}(m) \cap \cap_{j=1}^{m-1}[\cC_{n}(j)]^c\right\}\right]=\PP_{n}(\tau=n+m) / g_{n}(m) .
\$
Above all, for any $m\ge 1$, we have $\QQ_{n}(U=1,M=m)=\PP_{n}(\tau=n+m)$. 
\end{proof}

Based on Theorem \ref{thmA1}, if we obtain an output $I=0$ from Algorithm \ref{SNRB}, we claim that the record-breaker will not happen again after level $n$. Otherwise, i.e. if we obtain $I=1$ and associated $M=m$ from Algorithm \ref{SNRB}, we claim that the first next record-breaker happens at level $n+m$. In the later case, Algorithm \ref{ECM} also outputs a path leading to the next record-breaker. We next show that the output of Algorithm \ref{ECM} when $I=1$ is a realization of the fBM conditional on that the next record-breaker happens at level $n+m$. 

\begin{theorem}\label{thmA2}  
For the output of Algorithm \ref{ECM}, given $M=m$, $I=1$, 
the distribution of the values of fBM on the augmented points $D_{{n}+m}/D_{{n}}$, follows $\PP_{n}(\cdot|\tau={n}+m)$. 
\end{theorem}       

\begin{proof}
Let $\delta D^m_{n}=D_{{n}+m}/D_{{n}}$.
We consider a sequence of measurable sets $H_j$ where $j$ is chosen such that $t_j \in  \delta D^m_{n}$. We only need to show that
\$ 
\PP_{n}\bigl(B^H(t_j) \in H_j, t_j \in \delta D^m_{n} | \tau={n}+m\bigr)=\QQ_{n}\bigl(B^H(t_j) \in H_j, t_j \in \delta D^m_{n}\ \big |\ I=1, M=m\bigr).
\$
By definition, we have  
\$ 
& \PP_{n}\bigl(B^H(t_j) \in H_j, t_j \in \delta D^m_{n} | \tau={n}+m\bigr)={\PP_{n}\bigl(B^H(t_j) \in H_j, t_j \in   \delta D^m_{n},  \tau={n}+m\bigr)}\big / {\PP_{n}(\tau={n}+m)}\\
&\qquad = \EE_{{n}}\Bigl[ \prod_{t_j \in  \delta D^m_{n}}  1\left\{B^H(t_j) \in H_j\right\}  \cdot 1\left\{\tau={n}+m\right\}  \Bigr ] \Big/\PP_{n}(\tau={n}+m) .
\$
Similar to the proof of Theorem \ref{thmA1}, we have
\$
& \EE_{{n}}\Bigl[ \prod_{t_j \in  \delta D^m_{n}}  1\left\{B^H(t_j) \in H_j\right\}  \cdot 1\left\{\tau={n}+m\right\} \Bigr ]\\
&=\sum_{k=1}^{2^{{n}+m-1}}\EE_{{n}}\Bigl[ \prod_{t_j \in  \delta D^m_{n}}  1\{B^H(t_j) \in H_j\}  \cdot  R_{{n}+m}^{-1}\cdot 1\bigl\{\{ \bbeta^{\top} \balpha_n(m,k) > \ell_{{n}+m} \} \cap \cap_{j=1}^{m-1}[\cC_{n}(j)]^c\bigr\} \\
&\qquad +\prod_{t_j \in  \delta D^m_{n}}  1\{B^H(t_j) \in H_j\}  \cdot  R_{{n}+m}^{-1}\cdot 1 \bigl\{\{  \bbeta^{\top} \balpha_n(m,k) <- \ell_{{n}+m} \} \cap \cap_{j=1}^{m-1}[\cC_{n}(j)]^c\bigr\}\Bigr].
\$
The first term in above equation, corresponding to up-crossing record-breaker, and 
\$
&\sum_{k=1}^{2^{{n}+m-1}}\EE_{\QQ^{(m,k,+)}_{n}}\Bigl[\prod_{t_j \in  \delta D^m_{n}}  1\{B^H(t_j) \in H_j\} \\
&\qquad   \times  1\bigl\{\{  \bbeta^{\top} \balpha_n(m,k) > \ell_{{n}+m} \} \cap \cap_{j=1}^{m-1}[\cC_{n}(j)]^c\bigr\} \cdot  R_{{n}+m}^{-1}\cdot {\ud{\PP}_{n}}/{\ud\QQ^{(m,k,+)}_{n}} \Bigr]\\
&=\sum_{k=1}^{2^{{n}+m-1}}\EE_{\QQ^{(m,k,+)}_{n}}\Bigl[ \prod_{t_j \in  \delta D^m_{n}}  1\{B^H(t_j) \in H_j\} \\
&\qquad   \cdot 1\bigl\{\{  \bbeta^{\top} \balpha_n(m,k) > \ell_{{n}+m} \} \cap \cap_{j=1}^{m-1}[\cC_{n}(j)]^c\bigr\} \cdot  R_{{n}+m}^{-1}\cdot  \Theta^+_{n}(m,k) \Bigr]\cdot {g_{n(m)}}/{2^{n+m}}.
\$ 
Similar equation holds for the second term, corresponding to downward-crossing record-breaker.

On the other hand, we have 
\$
&\QQ_{n}\bigl(B^H(t_j) \in H_j, t_j \in \delta D^m_{n}\ \big |\ I=1, M=m\bigr)\\
&\qquad ={\QQ_{n}\bigl(B^H(t_j) \in H_j, t_j \in   \delta D^m_{n},  I=1\ \big | \ M=m \bigr)}\cdot   g_{n}(m) /  {\QQ_{n}(I=1, M=m)}
\$
We notice that
\$
& {\QQ_{n}\bigl(B^H(t_j) \in H_j, t_j \in   \delta D^m_{n},  I=1\ \big | \ M=m \bigr)}=\EE_{\QQ_{n}}\Bigl[1\{I=1\}\cdot \prod_{t_j \in  \delta D^m_{n}}  1\{B^H(t_j) \in H_j\} \ \big | \ M=m\Bigr ]\\
&=\sum_{k=1}^{2^{n+m-1}} \sum_{\pi=+,-} \EE_{\QQ^{(m,k,\pi)}_{n}}\Bigl[1\{I=1\} \prod_{t_j \in  \delta D^m_{n}}  1\{B^H(t_j) \in H_j\}\Bigr ]\cdot \frac{1}{2^{n+m}}.
\$
Now as
\$
&\EE_{\QQ^{(m,k,+)}_{n}}\Bigl[1\{I=1\} \ | \ \bB^H_{n+m}\Bigr ] \\
&\qquad =\EE_{\QQ^{(m,k,+)}_{n}}\Bigl [1\bigl\{\{   \bbeta^{\top} \balpha_n(m,k) > \ell_{{n}+m} \} \cap \cap_{j=1}^{m-1}[\cC_{n}(j)]^c\bigr\} \cdot \Theta^+_{n}(m,k) \cdot  R_{{n}+m}^{-1}\ \Big| \ \bB^H_{n+m}\Bigr],
\$
for the up-crossing part, we have
\$
&\sum_{k=1}^{2^{n+m}} \EE_{\QQ^{(m,k,+)}_{n}}\Bigl[1\{I=1\} \cdot \prod_{t_j \in  \delta D^m_{n}}  1\{B^H(t_j) \in H_j\}\Bigr ]=\sum_{k=1}^{2^{{n}+m-1}}\EE_{\QQ^{(m,k,+)}_{n}}\Bigl[ \prod_{t_j \in  \delta D^m_{n}}  1\{B^H(t_j) \in H_j\} \\
&\qquad   \times 1\bigl\{\{  \bbeta^{\top} \balpha_n(m,k) > \ell_{{n}+m} \} \cap \cap_{j=1}^{m-1}[\cC_{n}(j)]^c\bigr\} \cdot  R_{{n}+m}^{-1}\cdot  \Theta^+_{n}(m,k) \Bigr]\cdot \frac{1}{2^{n+m}}. 
\$ 
Similar result holds for the downward-crossing part. Finally, by Theorem \ref{thmA1}, $\PP_{n}(\tau={n}+m)=\QQ_{n}(U=1,M=m).$ Hence, we have 
\$ 
\PP_{n}\bigl(B^H(t_j) \in H_j, t_j \in \delta D^m_{n} | \tau={n}+m\bigr)=\QQ_{n}\bigl(B^H(t_j) \in H_j, t_j \in \delta D^m_{n}\ \big |\ U=1, M=m\bigr).
\$
\end{proof}

Theorem \ref{thmA1}, together with Theorem \ref{thmA2}, justifies the correctness of our algorithm.

\subsection{Computational complexity analysis} \label{cc}
In this section, we analyze the computational complexity\footnote{We refer to the computational complexity as the total number of uniform random variables we need to generate and the number of basic calculations.
For example, the Cholesky decomposition of an $n\times n$ convariance matrix has a computational complexity of $O(n^3)$.} of our algorithm. 
Note that the main component in our algorithm is to generate the last record-breaker and the associated path of the fBM.
Once we have found the last record-breaker, to achieve an $\epsilon$ error bound, the complexity is 
$O\left(2^{3N(\epsilon)}\right)$
if we use the naive Cholesky decomposition to sample the fBM at $D_{N(\epsilon)}/D_N$. 
Using a recursive construction of fractional Brownian bridge, we can reduce the complexity to 
$$O\left(2^{N(\epsilon)}\log\bigl(2^{N(\epsilon)}\bigr)\right)=O\left(\epsilon^{-1/(H-\delta)}\log\bigl(\epsilon^{-1/(H-\delta)}\bigr)\right).$$
We provide more details about this recursive construction towards the end of this section.

 We use $\bar N$ to denote the last level we need to generate in order to determine the level of the last record-breaker.
Note that in our algorithm, in order to apply the change-of-measure technique, we need to refine the dyadic approximation until the BCE condition is satisfied. Thus, $\bar N \geq N$. We also denote by
$ C_{\bar{N}} $ the associated computational complexity. This includes the cost of determining that 
there will be no more record-breakers.
The following theorem establishes that the computational cost is finite in expectation.

\begin{theorem} For the cost of generating the last record-breaker in Algorithm \ref{SLRB}, $ C_{\bar{N}}$, we have 
$$\EE[C_{\bar{N}}]<\infty.$$  
\end{theorem}
\begin{proof}
Recall that in our algorithm, the computational complexity arises from two main procedures. The first is finding the next record-breaker or claiming that there is no record-breakers any more. The second is refining dyadic approximation until the BCE condition is satisfied. 

Recall that $N$ is the time of the last record-breaker, which is also an upper bound for the number of while loops in Algorithm \ref{SLRB}.
We also recall that $\mathcal{N}_{\cE}$ is the last level the BCE condition is violated. Then we have
\$
\bar{N}\le \max\{N,\mathcal{N}_{\cE}+1\}.
\$

At current level $n$, we first analyze the computational complexity of checking the BCE condition. For now on, for simplicity, we use $C$ to denote a generic constant, which may differ from line to line. Based on Algorithm \ref{BCEC}, we need to calculate the conditional expectation until level $n+M_n=O(n+\log_2(\gamma_n))$, whose computational complexity is at most $O(2^{3n}\gamma_n^3).$ 
If the BCE condition is satisfied, then the computational complexity of applying the change-of-measure is of order
\$
\sum_{m=1}^{\infty}g_{n}(m)\cdot c_n(m)\le CZ_n^{-1} \cdot \sum_{m=1}^{\infty}2^{4({n}+m)}\cdot \exp\bigl\{-\rho^2/8\cdot 2^{2({n}+m)\delta} \bigr \} \le  C\cdot 2^{3n},
\$
where $c_n(m)=O(2^{3(n+m)})$, which denotes the computational complexity of Algorithm \ref{SNRB} conditional on $M=m$. 
If the BCE condition is broken, we refine the dyadic approximation until a level $ n^*\le \mathcal{N}_{\cE}+1$  where the BCE condition is satisfied. For each refinement level, we need to check the BCE condition. Hence, the corresponding computational complexity is $O(\sum_{j=n+1}^{n^*} 2^{3j}\gamma_j^3 ) .$ 

Above all, the total computational complexity for a while loop in Algorithm \ref{BCEC} can be upper bounded by 
\$C\cdot \left(  \sum_{j=1}^{\mathcal{N}_{\cE}+1} 2^{3j}\gamma_j^3+ 2^{3\bar{N}} \right),\$ 
As there are at most $\bar{N}$ loops.We have 
\# \label{ccb}
C_{\bar{N}} \le C\cdot \left(  \sum_{j=1}^{\mathcal{N}_{\cE}+1} 2^{3j}\gamma_j^3+ 2^{3\bar{N}} \right)\cdot \bar N
=C\sum_{j=1}^{\mathcal{N}_{\cE}+1} 2^{3j}\gamma_j^3 \cdot \bar N + C2^{3\bar{N}}\cdot \bar N.
\# 
Since $N$ and $\mathcal{N}_{\cE}$ have finite moment generating function, so it is $\bar N$. Then the second term in \eqref{ccb} in bounded in expectation. 
We next establish an upper bound for the first term. By Cauchy-Schwartz inequality, we have 
\$
\EE\left[\sum_{j=1}^{\mathcal{N}_{\cE}+1} 2^{3j}\gamma_j^3 \cdot \bar N\right] &\le \biggl( {\EE\biggl[ \Bigl(\sum_{j=1}^{\mathcal{N}_{\cE}+1}2^{6j}\Bigr)\cdot \Bigl(\sum_{j=1}^{\mathcal{N}_{\cE}+1}\gamma_j^6 \Bigr) \biggr]}\biggr)^{1/2}\cdot \bigl({\EE[\bar N^2]}\bigr)^{1/2} \\
&\le C \cdot \bigl( {\EE[ \bar{N}^2]}\bigr)^{1/2}\cdot \left({\EE\left[ \sum_{j=1}^{\mathcal{N}_{\cE}+1}2^{6\mathcal{N}_{\cE}} \cdot \gamma_j^6 \right]}\right)^{1/2}.
\$
For the last term in above inequality, recall that $ \gamma_n$ is the maximal absolute value of the entries in the vector $ \bSigma_{n}^{-1}\bB^H_{n} $, which follows multivariate normal distribution $N(0, \bSigma_{n}^{-1})$. By Fubini's Theorem and Cauchy-Schwartz inequality, we have 
\$
\EE\left[ \sum_{j=1}^{\mathcal{N}_{\cE}+1}2^{6\mathcal{N}_{\cE}} \cdot \gamma_j^6 \right]&=\EE\left[\sum_{j=1}^{\infty}2^{6\mathcal{N}_{\cE}}\cdot \gamma_j^6 \cdot 1\{j\le \mathcal{N}_{\cE}+1\} \right]\\
&=\sum_{j=1}^{\infty}\EE\bigl[2^{6\mathcal{N}_{\cE}} \cdot \gamma_j^6 \cdot 1\{j\le \mathcal{N}_{\cE}+1\}\bigr] \le \Bigl(\EE\bigl[2^{12\mathcal{N}_{\cE}}]\Bigr)^{1/2}\cdot   \sum_{j=1}^{\infty} \Bigl(\EE \bigl[ \gamma_j^{12} \cdot 1\{j\le \mathcal{N}_{\cE}+1\}\bigr]\Bigr)^{1/2} \\
&\le \Bigl(\EE\bigl[2^{12\mathcal{N}_{\cE}}]\Bigr)^{1/2}\cdot  \sum_{j=1}^{\infty} (\EE[\gamma_j^{24} ])^{1/4}\cdot \bigl( \PP(\mathcal{N}_{\cE}\ge j-1)\bigr)^{1/4}.
\$ 
It is easy to see that $\EE[\gamma_j^{24} ] \le C\cdot 2^j\cdot 2^{24Hj}=C\cdot 2^{(1+24H)j}$. Similar to the proof of Theorem \ref{ucr}, using the decay rate of $\PP(\mathcal{N}_{\cE}\ge j)$ proved in Lemma \ref{finitebce} , we obtain $ \EE[ \sum_{j=1}^{\mathcal{N}_{\cE}+1}2^{6\mathcal{N}_{\cE}}\cdot \gamma_j^6 ]<\infty. $ Hence, $\EE[C_{\bar{N}}]<\infty$, i.e. our algorithm has finite expected computational complexity.
\end{proof}

In our algorithm, once we have found the level of the last record-breaker $N$, we only need to refine the dyadic approximation until the desired truncation level $N(\epsilon)$, conditioning on that the record-breakers do not happen beyond $N$. Note that $N$ does not depend on $\epsilon$. Thus, $N$ can be treated as a constant. Assuming that $ N(\epsilon)>N  $, then we need to sample the fBM at time points $D_{N(\epsilon)}/D_N$ conditional on $\bB^H_N$, which involves $O(2^{N(\epsilon)})$ correlated Gaussian random variables. If we sample naively, i.e., calculating the conditional distribution first and sample from it, the computational complexity is $O(2^{3N(\epsilon)})$. This complexity can be reduced to $O(2^{N(\epsilon)}\log(2^{N(\epsilon)}))$ using a recursive construction of Gaussian bridge. We next introduce the details of this recursive construction. The algorithm is based on the Gaussian bridge construction developed in \cite{sottinen2014generalized}. We summarize the main idea in the following lemma.
\begin{lemma} \label{gb}
 Let $\{X_t\}_{t\ge 0}$ be a stationary Gaussian process with covariance function $r(s,t)$. Then the distribution of $X_t$ conditional on that $X_{t_k}=y_k, k=1,\cdots,n $ is same with that of $X^n_t$, which is defined recursively as 
\$
X^0_t&=X_t \\
X^k_t&=X^{k-1}_t-\frac{r_{k-1}(t,t_k)}{r_{k-1}(t_k,t_k)}\cdot \bigl(X^{k-1}_{t_k}-y_k\bigr),\ k=1,\cdots,n
\$
where 
\$
r_0(s,t)&=r(s,t) \\
r_k(s,t)&=r_{k-1}(s,t)-\frac{r_{k-1}(s,t_k)\cdot r_{k-1}(t_k,t)}{r_{k-1}(t_k,t_k)},\ k=1,\cdots,n
\$
\end{lemma} 
 In our case, to sample the fBM at $D_{N(\epsilon)}/D_N$ conditional on $\bB^H_N$, we first sample without taking condition using the Davies and Harte method (\cite{davies1987tests}), which has complexity $ O(2^{N(\epsilon)}\log(2^{N(\epsilon)}))$. Then, for each $t \in D_{N(\epsilon)}/D_N$,  we use the recursion in Lemma \ref{gb} to calculate the value of fBM at $t$ after taking condition for $2^N$ steps whose complexity is approximately $O(2^{3N})$. We repeat this procedure $O(2^{N(\epsilon)})$ times. Above all, the total computational complexity of refine the dyadic approximation from level $N$ to $N(\epsilon)$ is $ O(2^{N(\epsilon)}\log(2^{N(\epsilon)}))$.

\section{$\epsilon$-Strong simulation of stochastic differential equation driven by fBM} \label{sec:sde}
In this section, we extend the $\epsilon$-strong simulation algorithm for fBM to stochastic differential equations (SDEs) driven by fBM with $H>1/2$ via rough path theory.
Consider a $d$-dimensional SDE 
\# \label{fsde}
\ud \bY(t)=\bmu(\bY(t)) \ud t+\bsigma(\bY(t))\ud \bB^H(t),\ \bY(0)=\by(0),
\#  
where $\bB^H(t) $ is a $d^{\prime}$-dimensional fBM (each component is an independent standard one-dimensional fBM),  $\bmu(\cdot): \RR^{d}\to \RR^{d}$ and $\bsigma(\cdot): \RR^{d} \to \RR^{d \times d^{\prime}}$ are driven vector fields, corresponding to the drift and the volatility, respectively. For any fixed $\epsilon>0$, our goal is to construct a probability space, supporting both $\bY(t)$ and a fully simulatable path $\hat\bY_{\epsilon}(t)$ such that
\$
\sup_{t\in [0,1]} |\bY(t) - \hat \bY_{\epsilon}(t)| \le \epsilon, \quad \text{a.s.}
\$ 

The construction of $\hat{\bY}_{\epsilon}(t)$ builds on our ability to estimate the driving fBM and its corresponding $\alpha$-H\"older norm. 
Particularly, for any $1/2<\alpha<H$, the sample path of fBM is $\alpha$-H\"older continuous almost surely. In this case, by the rough path theory, the solution of SDE \eqref{fsde} can be defined path by path and the mapping from $\bB^H$ to $\bY$ is continuous under the $\alpha$-H\"older norm \citep{lyons1998differential} \footnote{For $H\le1/2$, to define the corresponding SDE in a path by path sense, high order iterated integrals of $\bB^H$ need to be specified.}. 
Therefore, if we can control the error of the simulated driving signals, by continuous mapping type of argument, we can also control the error of the simulated SDEs.

In what follows, we shall first lay out the main idea of our algorithmic development. We then present the theoretical foundation and derivation in Section \ref{sec:rough_path}.
The construction of the approximated solution is based on simple Euler scheme. For dyadic discretization $D_n$,
we define $\bY_n(t^n_{k})$, $k=1,2,\dots, 2^n$, via the recursion
\$
\bY_n(t^n_{0})&=0,\\
\bY_n(t^n_{k+1})&=\bY^{(n)}(t^n_{k})+\bmu\bigl(\bY_n(t^n_{k})\bigr)\cdot \Delta_n+\bsigma\bigl(\bY_n(t^n_{k})\bigr)\cdot (\bB^H(t^n_{k+1})-\bB^H(t^n_k)).
\$

We then construct the whole path $\bY_n(t)$ via piecewise-constant interpolation, 
\$
\bY_n(t)=\bY_n(t^n_{k}) ,\ t\in [t^n_{k},t^n_{k+1}).
\$ 
The challenge here is to choose an appropriate discretization level $N_Y(\epsilon)$, such that
\$
 \|  \bY_{N_Y(\epsilon)}-\bY\|_{\infty} \le \epsilon, \quad \text{a.s.}
\$ 
In Theorem \ref{thm0} below, we establish that 
\begin{equation} \label{eq:thm0}
\|\bY_n-\bY  \|_{\infty} \le G\cdot \Delta_n^{2\alpha-1},
\end{equation}
where $G$ can be characterized explicitly, and it depends on the $\alpha$-H\"older norm of $\bB^H$. 
Therefore, we shall first use the $\epsilon$-strong simulation algorithm we developed in Section \ref{sec:alg} to find an upper bound for 
$||\bB^H||_{\alpha}$. We can then upper bound $G$, and set 
$$N_Y(\epsilon)=\lceil \log_2(\epsilon^{-1}G/(2\alpha-1))\rceil.$$
The actual algorithm is summarized in Algorithm \ref{tem}. 

\begin{algorithm} 
\caption{$\epsilon$-Strong Simulation of SDE Driven by fBM (SSDE)}
\begin{algorithmic}[1]
\STATE \textbf{Input:} Accuracy $\epsilon$,  vector fields $\bmu, \bsigma$,, H\"older norm order $\alpha$, Hurst index $H$, record-breaker parameters $\nu,\rho,\delta$. 
\STATE \textbf{Estimate $\alpha$-H\"older norm :}  
\STATE  \qquad    \textbf{For $i$ in $[d^{\prime}]$:}
\STATE  \qquad \qquad    Call SLRB (Algorithm 2): set $[\hat{B}^i, N^i]\leftarrow \text{SLRB}(H,\nu,\rho,\delta)$. 
\STATE  \qquad \qquad   Calculate the upper bound for the $\alpha$-H\"older norm:
$$\hat{C}^i_{\alpha}\leftarrow \|\hat{B}^i\|_{\alpha}+\frac{\rho2^{2-\alpha}\cdot 2^{-(H-\alpha-\delta)(N^i+1)}}{1-2^{-(H-\alpha-\delta)}}.$$
\vspace{-5mm}
\STATE  \qquad Set $\hat{C}_{\alpha}=\max_{i \in [d^{\prime}]}\hat{C}^i_{\alpha} \wedge 1 $ . 
\STATE \textbf{Determine the truncation level:} 
\STATE  \qquad Calculate $G$ using $\hat{C}_{\alpha}$, then set $N_{Y}\leftarrow \lceil \log_2(\epsilon^{-1}G/(2\alpha-1))\rceil$.  
\STATE \textbf{Refine the approximation of fBM:} 
\STATE  \qquad  \textbf{For $i$ in $[d^{\prime}]$:}
\STATE \qquad \qquad  \textbf{If $N^i<N_Y$:} refine dyadic approximation $\hat{B}^i$ until level $N_Y$ via acceptance-rejection. 
\STATE \textbf{Solve the SDE by Euler scheme:} 
\STATE  \qquad    $\bY_{N_Y}(t^{N_Y}_{0})\leftarrow 0$ 
\STATE  \qquad   \textbf{For $k$ in $[2^{N_Y}]$:} 
\STATE \qquad \qquad  $\bY_{N_Y}(t^{N_Y}_{k+1})\leftarrow \bY_{N_Y}(t^{N_Y}_{k})+\bmu\bigl(\bY_{N_Y}(t^{N_Y}_{k})\bigr)\Delta_{N_Y}+\bsigma\bigl(\bY_{N_Y}(t^{N_Y}_{k})\bigr) (\hat{\bB}^H(t^{N_Y}_{k+1})-\hat{\bB}^H(t^{N_Y}_k)). $
\STATE \textbf{Output:} Return $\bY_{N_Y}(t) $, the piecewise constant interpolation of $\{\bY_{N_Y}(t^{N_Y}_{k}) \}_{k=0,\cdots,2^{N_Y}}$. 
\end{algorithmic}\label{tem} 
\end{algorithm}

\subsection{Rate of convergence of Euler scheme} \label{sec:rough_path}
In this section, we present the details of the rate of convergence of Euler scheme. This result is an extension of \cite{lejay2010controlled}. Particularly, we explicitly characterize the constant in front of $\Delta_{n}^{2\alpha-1}$ for the Euler scheme at dyadic discretization $D_n$. This is important for our algorithmic development, as we need to know $G$ in \eqref{eq:thm0} to find the required discretization level $N_Y$.      
           
We first introduce a few notations to simplify the exposition. 
Consider the following multidimensional ordinary differential equation (ODE) system $y$ driven by vector-valued signal $x$ 
\# \label{ode}
\left[
 \begin{matrix}
  \ud y^{1}(t) \\
  \ud y^{2}(t) \\
   \vdots \\
   \ud y^{d}(t)\\
  \end{matrix}
  \right]=
  \left[
  \begin{matrix}
  f_{11}(\by(t))& f_{12}(\by(t))& \cdots & f_{1h}(\by(t)) \\
   f_{21}(\by(t))& f_{22}(\by(t))& \cdots & f_{2h}(\by(t)) \\
   \vdots&  \vdots&  & \vdots \\
   f_{d1}(\by(t))& f_{d2}(\by(t))& \cdots & f_{dh}(\by(t)) \\
     \end{matrix}
  \right]\cdot \left[
 \begin{matrix}
  \ud x^{1}(t) \\
  \ud x^{2}(t) \\
   \vdots \\
   \ud x^{h}(t)\\
  \end{matrix}
  \right],\quad
   \left[
   \begin{matrix}
  y^{1}(0) \\
  y^{2}(0) \\
   \vdots \\
   y^{d}(0)\\
  \end{matrix}
  \right]=\left[
 \begin{matrix}
  y^{1}_0 \\
  y^{2}_0 \\
   \vdots \\
   y^{d}_0\\
  \end{matrix}
  \right],
\#
where $\by(t)=[y^{1}(t),\cdots,y^{d}(t)]^{\top}$ and $ \bx(t)=[x^{1}(t),\cdots,x^{h}(t)]^{\top} $. It is easy to see that SDE \eqref{fsde} can be written in form of \eqref{ode}. We only need to set 
\[
\bbf(\cdot)=\left[\begin{array}{c}
\bmu(\cdot)\\
\bsigma(\cdot) 
\end{array}\right]
\mbox{ and }
\ud \bx(t)=\left[\begin{array}{c}
\ud t\\
\ud \bB^H(t)
\end{array}\right],
\] 
and then dimension $h=d'+1$. From now on, all of our derivation will be based on the notations in \eqref{ode}. 
In the following of this section, we use $ \bbf$ to denote the matrix $[f_{ij}]_{d\times h}$. 
Note that we will use bold letters to denote matrices or vectors. 
Furthermore, we assume that $ \bx $ is $\alpha$-H\"older continuous with $1/2<\alpha\le1,$ which is to say, $|\bx(t)-\bx(s)|\le C_{\alpha}|s-t|^{\alpha}$ for some $C_{\alpha}\in (0,\infty)$.

The solution to equation \eqref{ode} is formally defined in terms of the Young integral. In what follows,
we shall start by a brief introduction to Young integral and then quantify error of the Euler approximation scheme.
\begin{definition} \label{sol}
We say that $\by(t)$ is a solution of equation \eqref{ode}  if for all $t \in [0,1]$,
\$
\by(t)=\by_0+\int_0^t \bbf\bigl(\by(s)\bigr)\ud \bx(s),
\$
where the integration is interpreted as Young integral.
\end{definition}
Young integral is an extension of Riemann-Stieltjes integral for paths with finite $p$-variation, $1<p<2$, but potentially infinite total variation. 
Recall that we call a continuous path $u(t)$ defined on $[0,1]$ has finite $p$-variation if 
\$
\sup_{\Pi}\sum_{t_i \in \Pi} | u(t_{i+1})-u(t_{i}) |^p <\infty,
\$
where $\Pi=\{t_i\}_{i\ge 0}$ is a set of finite partitions of $[0,1]$.
Then we have
 \begin{definition}(Young integral) Let $u(t)$ and $v(t)$ be continuous paths on $[0,1]$ with finite $p$-variation and $q$-variation respectively, such that 
 \$
 \frac{1}{p}+\frac{1}{q}>1.
 \$ 
Then the limit of Riemann sum as the mesh of the partition $|\Pi|$ goes to zero
\$
\lim_{|\Pi| \to 0}\sum_{t_i \in \Pi} v(t_i)\cdot(  u(t_{i+1})-u(t_{i})),
\$  
exists and is unique. We use $\int_0^1 v(s)\ud u(s),$ to denote this limit and call it the Young integral of $v(s)$ with respect to $u(s)$. 
 \end{definition}


A special case of finite $p$-variation path is $\alpha$-H\"older continuous path. Note that if $u$ is $\alpha$-H\"older continuous, i.e. $| u(s)-u(t)| \le  \|u\|_{\alpha} \cdot |s-t|^{\alpha}$ for some $ \|u\|_{\alpha} \in(0,\infty)$, then $u$ has finite $p$-variation with $p=1/\alpha$. From now on, we use $H_{\alpha,[0,1]}(u)$ to denote the $\alpha$-H\"older norm of $u$ on interval $[0,1]$. In this case, we have the following Young-L\'oeve estimate (\cite{lejay2010controlled}).
\begin{theorem} (Young-L\'oeve estimate.) Assume that the integrator $u$ and the integrand $v$ are H\"older continuous of  exponents $\alpha$ and $\beta$ with $ \alpha+\beta>1$, respectively. Then for any $s,t \in [0,1]$,  
\$
\Bigl| \int _s^t v(r)\ud u(r)- v(s)(u(t)-u(s))\Bigr| \le K(\alpha+\beta)\cdot H_{\alpha,[0,1]}(u)\cdot H_{\beta,[0,1]}(v)\cdot |t-s|^{\alpha+\beta},
\$
where $ K(\alpha+\beta)=1+\sum_{n\ge 1} n^{-(\alpha+\beta)} $. Moreover, for any finite partition $\Pi_{s,t}$ of $[s,t]$,
\$
\Bigl|\sum_{t_i \in \Pi_{s,t}} v(t_i)(u(t_{i+1})-u(t_i))- v(s)(u(t)-u(s))\Bigr| \le K(\alpha+\beta)\cdot H_{\alpha,[0,1]}(u)\cdot H_{\beta,[0,1]}(v)\cdot |t-s|^{\alpha+\beta},
\$
\end{theorem}

We are ready to introduce the Euler scheme. For  dyadic discretization points
$D_n=\{t^n_k\}_{k=0,\cdots,2^{n}}$, we define 
\#  \label{re1}
\by_n(t^n_{k+1})=\by_n(t^n_{k})+\bbf\bigl(\by_n(t^n_{k})\bigr)\cdot (\bx(t^n_{k+1})-\bx(t^n_k)),\ k=0,1,\cdots,2^n-1.
\#  
Based on the values $\{\by_n(t^n_{k})\}_{k=0,\cdots,2^n}$, we further construct a continuous path over $[0,1]$ via piecewise constant interpolation, i.e.
\# \label{li}
\by_n(t)=\by_n(t^n_{k}) ,\ t\in [t^n_{k},t^n_{k+1}).
\#   
Then we call $\by_n(t)$ an approximated solution of level $n$ via Euler scheme. 
Our goal is to control the uniform norm between the approximated solution $\by_n(t) $ and exact solution $\by(t)$.

To ensure the existence of solution of ODE \eqref{ode} and control the approximation error, we impose the following smoothness condition on the vector field $\bbf$. 
\begin{assumption}\label{ass0}
We assume that $\bbf(\cdot) \in C^2(\RR^d)$ and $\max\{ |\bbf|, |\nabla \bbf|, |\nabla^2 \bbf| \}<\infty, $ where $| \bbf| =\max_{ij} | f_{ij}| ,$ $|\nabla \bbf| =\max_{ij} |\nabla f_{ij}| $, and  $|\nabla^2 \bbf| =\max_{ij} |\nabla^2 f_{ij}| .$
\end{assumption}

We also define the following constants. Let
\$
G^*_1&=2h \cdot \lceil (2dhC_{\alpha}K(2\alpha) |\nabla \bbf|)^{1/\alpha} \rceil^{1-\alpha}\cdot |\bbf|\cdot C_{\alpha} ,\qquad 
G^*_2= dh\cdot K(2\alpha)\cdot |\nabla \bbf| \cdot  C_{\alpha}\cdot G^*_1,\\
L &=4/(1-2^{1-2\alpha})\cdot (hC_{\alpha})^{2} |\nabla\bbf |\cdot  |\bbf |,\qquad \omega=(h|\bbf |C_{\alpha}/L)^{1/\alpha},\\
G_1&=\bigl(L+h |\bbf |C_{\alpha}\bigr)\cdot (1+ \omega^{-1}) ,\qquad 
G_2= \max\Bigl\{  (2\omega^{-\alpha}+\omega^{-1-\alpha})\cdot \big(L+h|\bbf|C_{\alpha} \bigr), L\Bigr\}.
\$
 In addition, we define a sequence of constants $\{ \Gamma_{k}\}$ and $\{ \Upsilon_{k}\}$ via recursion 
\begin{eqnarray} \label{eq:Upsilon}
&&\Gamma_{1}=2G_2^*,\ \Upsilon_{1}= (4\zeta)^{-1}\cdot \Gamma_{1},  \nonumber\\
&&\Gamma_k=2(G_2^*+\upsilon\cdot  \Upsilon_{k-1}) ,\ \Upsilon_{k}=(4\zeta)^{-1}\cdot \Gamma_{k}+\Upsilon_{k-1} ,\ k\ge2, 
\end{eqnarray}
where 
\$
\zeta&=h\cdot K(2\alpha)\cdot C_{\alpha}\cdot \bigl(  d\cdot |\nabla \bbf|+d^2 \cdot |\nabla^2 \bbf|\cdot ( G_1^* +G_1 ) \bigr),\\
\upsilon&= C_{\alpha}\cdot \bigl(d^2h\cdot K(2\alpha)\cdot |\nabla^2\bbf|\cdot( G_1^* +G_1 )+d\cdot |\nabla \bbf| \bigr).
\$  

Under Assumption \ref{ass0}, the general theory of Young integral equation ensures the existence and uniqueness for the solution of equation \eqref{ode}. The following theorem characterizes the rate of convergence of the Euler scheme under the uniform norm and is the main result of this section. 
\begin{theorem} \label{thm0}Under Assumption \ref{ass0}, 
\$
\|\by_n(t )-\by(t )\|_{\infty} \le G\cdot  \Delta_n^{2\alpha-1}:=\bigl( \Upsilon_{\lceil (4\zeta)^{1/\alpha} \rceil}+G^*_1 \bigr) \cdot  \Delta_n^{2\alpha-1}.
\$ 
 \end{theorem} 
 
\begin{remark}
By the definition of $G$, it is easy to see that $G$ is an increasing function with respect to $C_{\alpha}$, the $\alpha$-H\"older norm of $\bx$.
\end{remark}

\begin{proof}
The proof of Theorem \ref{thm0} relies on the several lemmas. The first one establishes the existence of the solution to equation \eqref{ode} and its properties. 
\begin{lemma} \label{lem0.0} 
Under Assumption \ref{ass0}, the solution of equation \eqref{ode} $\by(t)  $ exists. We also have the following estimates. For all $0\le s<t\le 1$,  
\$
&|\by(s)-\by(t)|\le  G^*_1 \cdot |s-t|^{\alpha},\\ 
&\bigl |\by(t)-\by(s)-\bbf(\by(s))\cdot (\bx(t)-\bx(s))\bigl |\le G^*_2 \cdot  |s-t|^{2\alpha}.
\$
\end{lemma}

The next lemma is a counterpart of Lemma \ref{lem0.0}. It establishes similar properties for the solutions obtain by the Euler scheme.  

\begin{lemma} \label{lem0.1}
For all the dyadic discretization time points $t^n_{j},t^n_{r} \in [0,1]$, we have the estimates
\$
&|\by_n({t^n_j})-\by^{(n)}({t^n_r})|\le  G_1 \cdot \bigl|t^n_{j}-t^n_{r}\bigr|^{\alpha}\\ 
 &| \by_n({t^n_j})-\by^{(n)}({t^n_r})-\bbf(\by_n({t^n_r}))\cdot (\bx(t^n_j)-\bx(t^n_r))| \le G_2 \cdot   \bigr|t^n_{j}-t^n_{r}\bigr|^{2\alpha}.
\$    
\end{lemma}
We also need to define a restricted $\alpha$-H\"older norm. Specifically, given the dyadic partition $D_n$ and a path $\bx(t)$, the restricted $\alpha$-H\"older norm on $D_n$ is defined as 
\$
H_{\alpha}(\bx| D_n)=\sup_{0\le i< j \le 2^n} \frac{\bigl| \bx(t^n_i)-\bx(t^n_j) \bigr |}{|t^n_i-t^n_j |^{\alpha}}.
\$   
Note that in the restricted $\alpha$-H\"older norm, we do not require that the path is well-defined on points outside of $D_n$. Intuitively, the restricted $\alpha$-H\"older norm measures the $\alpha$-H\"older continuity of the solution obtained via Euler scheme on $D_n$. We have the following lemma on the restricted $\alpha$-H\"older norm of $  \bbf(\by)-\bbf(\by_n)$ on $D_n$. 
\begin{lemma} \label{lem0.2}
For dyadic discretization $D_n$, 
\$
H_{\alpha}\bigl( \bbf(\by)-\bbf(\by_n) | D_n\bigr)& \le \bigl(  d\cdot |\nabla \bbf|+ d^2 \cdot|\nabla^2 \bbf|\cdot ( G_1^* +G_1 ) \bigr)\cdot H_{\alpha}\bigl(\by-\by_n| D_n\bigr)+\\
&\qquad \qquad   d^2 \cdot |\nabla^2 \bbf|\cdot ( G_1^* +G_1 )\cdot |\by(0)-\by_n(0)|.
\$
\end{lemma}

With the above lemmas, we are ready to prove Theorem \ref{thm0}. 
Let 
$$\cJ_i\equiv\int_{t^n_i}^{t^n_{i+1}} \bigl( \bbf(\by(s))- \bbf(\by(t^n_i)) \bigr) \ud \bx(s).$$
By definition of Euler scheme and the solution of equation \ref{ode}, we have 
\$
\by_n(t^n_{i+1})-\by(t^n_{i+1})=\by_n(t^n_{i})-\by(t^n_{i})+ \Bigl( \bbf\bigl( \by_n(t^n_{i})  \bigr)- \bbf\bigl( \by(t^n_{i})  \bigr)   \Bigr)\cdot  (\bx(t^n_{i+1})-\bx(t^n_{i}))-\cJ_i,
\$
and furthermore, for all $0\le \ell < k\le 2^n$,
\$
\bigl( \by_n(t^n_{k})-\by(t^n_{k}) \bigr) -\bigl( \by_n(t^n_{\ell})-\by(t^n_{\ell}) \bigr) =\sum_{i=\ell}^{k-1} \Bigl( \bbf\bigl( \by_n(t^n_{i})  \bigr)- \bbf\bigl( \by(t^n_{i})  \bigr)   \Bigr)\cdot  (\bx(t^n_{i+1})-  \bx(t^n_{i}))-\sum_{i=\ell}^{k-1} \cJ_i.  
\$
Using the Young-L\'oeve estimate, we have the following bounds 
\$
&\biggl| \sum_{i=\ell}^{k-1} \Bigl( \bbf\bigl( \by_n(t^n_{i})  \bigr)- \bbf\bigl( \by(t^n_{i})  \bigr)   \Bigr)\cdot  (\bx(t^n_{i+1})-  \bx(t^n_{i}))- \Bigl( \bbf\bigl( \by_n(t^n_{\ell})  \bigr)- \bbf\bigl( \by(t^n_{\ell})  \bigr)   \Bigr)\cdot  (\bx(t^n_{k})-  \bx(t^n_{\ell}))\biggr| \\
&\qquad \qquad \le h\cdot K(2\alpha)\cdot H_{\alpha}\bigl(\bbf(\by)-\bbf(\by_n) | D_n\bigr) \cdot C_{\alpha}\cdot |t^n_k-t^n_{\ell}|^{2\alpha},
\$ 
where $ K(2\alpha)=1+\sum_{n=1}^{\infty} {n^{-2\alpha}}$ and $C_{\alpha}$ is the $\alpha$-H\"older norm of $\bx$. By Lemma \ref{lem0.0}, we also have 
\$
|\cJ_i|\le G_2^*\cdot |t^n_{i+1}- t^n_{i} |^{2\alpha}.
\$
Note that \$\sum_{i=\ell}^{k-1} |t^n_{i+1}- t^n_{i} |^{2\alpha} \le  |t^n_k-t^n_{\ell}|\cdot \Delta_n^{2\alpha-1} \le  |t^n_k-t^n_{\ell}|^{\alpha}\cdot \Delta_n^{2\alpha-1}.\$ 
Then we have
\$
& \Bigl| \bigl( \by_n(t^n_{k})-\by(t^n_{k}) \bigr) -\bigl( \by_n(t^n_{\ell})-\by(t^n_{\ell}) \bigr)\Bigr| \\
& \qquad \qquad \le h\cdot K(2\alpha)\cdot H_{\alpha}\bigl(\bbf(\by)-\bbf(\by_n) | D_n\bigr) \cdot C_{\alpha}\cdot |t^n_k-t^n_{\ell}|^{2\alpha}+ \\ & \qquad \qquad \Bigl( \bbf\bigl( \by_n(t^n_{\ell})  \bigr)- \bbf\bigl( \by(t^n_{\ell})  \bigr)   \Bigr)\cdot  (\bx(t^n_{k})-  \bx(t^n_{\ell}))+G^*_2\cdot |t^n_k-t^n_{\ell}|^{\alpha}\cdot \Delta_n^{2\alpha-1}.
\$
By the definition of restricted $\alpha$-H\"older norm, 
\$
  |\bbf\bigl( \by_n(t^n_{\ell})  \bigr)- \bbf\bigl( \by(t^n_{\ell})  \bigr)|  \le  H_{\alpha}\bigl(\bbf(\by)-\bbf(\by_n) | D_n\bigr) \cdot |t^n_{\ell}|^{\alpha}+  \bigl |\bbf( \by_n(0)  \bigr)- \bbf\bigl( \by(0)  )\bigr |,
\$  
combined with Lemma \ref{lem0.2}, we have 
\$
 H_{\alpha}\bigl(\by-\by_n | D_n\bigr)&  \le h\cdot K(2\alpha)\cdot ( |t^n_k-t^n_{\ell}|^{\alpha}+|  t^n_{\ell} |^{\alpha})  \cdot C_{\alpha}\cdot H_{\alpha}\bigl(\bbf(\by)-\bbf(\by_n) | D_n\bigr)+\\
&\qquad \qquad \qquad  G^*_2\cdot \Delta_n^{2\alpha-1}+C_{\alpha}\cdot |\bbf\bigl( \by_n(0)  \bigr)- \bbf\bigl( \by(0)  \bigr)|\\
&\le h\cdot K(2\alpha)\cdot C_{\alpha}\cdot \bigl( d\cdot  |\nabla \bbf|+ d^2\cdot |\nabla^2 \bbf|\cdot ( G_1^* +G_1 ) \bigr)\cdot ( |t^n_k-t^n_{\ell}|^{\alpha}+|  t^n_{\ell} |^{\alpha})  \\ 
& \times H_{\alpha}\bigl(\by-\by_n| D_n\bigr)+G^*_2\cdot \Delta_n^{2\alpha-1}+  C_{\alpha}\cdot \bigl( d^2h\cdot K(2\alpha)\cdot |\nabla^2\bbf|\cdot( G_1^* +G_1 )+d\cdot |\nabla \bbf| \bigr)\\
&\times | \by_n(0)- \by(0) |.
\$
By the definition of $\zeta$ and  $\upsilon$ in \eqref{eq:Upsilon},
\$
H_{\alpha}\bigl(\by-\by_n | D_n\bigr) \le \zeta ( |t^n_k-t^n_{\ell}|^{\alpha}+|  t^n_{\ell} |^{\alpha})\cdot H_{\alpha}\bigl(\by-\by_n | D_n\bigr)+G^*_2\cdot \Delta_n^{2\alpha-1}+\upsilon  | \by_n(0)- \by(0) |.
\$ 
Let $T_0=0$ and $T_1= \max\{t \in D_n, t\le (4\zeta)^{-1/\alpha}\} $. Then 
\$
H_{\alpha}\bigl(\by-\by_n | D_n \cap [T_0,T_1] \bigr) \le 2G^*_2\cdot \Delta_n^{2\alpha-1}+2\upsilon \cdot| \by_n(0)- \by(0) |.
\$  
Since the Euler scheme and the exact solution have same initial value, $| \by_n(0)- \by(0) |=0,$ then we have  
\$
H_{\alpha}\bigl(\by-\by_n | D_n \cap [T_0,T_1] \bigr)  \le 2G^*_2\cdot \Delta_n^{2\alpha-1},
\$
and furthermore,
\$
| \by_n(T_1)  -  \by(T_1) |\le 2G^*_2\cdot (4\zeta)^{-1}\cdot \Delta_n^{2\alpha-1}.
\$
Now we let $T_2= \max\{t \in D_n, t\le 2 (4\zeta)^{-1/\alpha}\}$. For $T_1\le t^n_{\ell} < t^n_k \le T_2$,  we have  
\$
H_{\alpha}\bigl(\by-\by_n | D_n\bigr) \le \zeta ( |t^n_k-t^n_{\ell}|^{\alpha}+|  t^n_{\ell}-T_1 |^{\alpha})\cdot H_{\alpha}\bigl(\by-\by_n | D_n\bigr)+G^*_2\cdot \Delta_n^{2\alpha-1}+\upsilon  | \by_n(T_1)- \by(T_2) |.
\$ 
As a result,
\$
H_{\alpha}\bigl(\by-\by_n | D_n \cap [T_1,T_2] \bigr) \le 2(G_2^*+\upsilon\cdot 2G^*_2\cdot (4\zeta)^{-1})\cdot  \Delta_n^{2\alpha-1},
\$ 
and 
\$
| \by_n(T_2)  -  \by(T_2) |\le   \Bigl[ 2(G_2^*+\upsilon\cdot 2G^*_2\cdot (4\zeta)^{-1})\cdot (4\zeta)^{-1}+2G^*_2\cdot (4\zeta)^{-1} \Bigr]   \cdot  \Delta_n^{2\alpha-1}. 
\$
We need to repeat this procedure at most $k^*= \lceil (4\zeta)^{1/\alpha} \rceil$ times in order to cover the whole interval $[0,1]$ and we can obtain the a sequence of bounds  
\$
H_{\alpha}\bigl(\by-\by_n | D_n \cap [T_{k-1},T_{k}] \bigr) &\le \Gamma_{k} \cdot  \Delta_n^{2\alpha-1},\\
| \by_n(T_{k})  -  \by(T_{k})  |& \le \Upsilon_{k}  \cdot  \Delta_n^{2\alpha-1}, 
\$
where $\{ \Gamma_{k}\}$ and $\{ \Upsilon_{k}\}$ are defined via recursion \eqref{eq:Upsilon}.
Note that $\{ \Gamma_{k}\}$ and $\{ \Upsilon_{k}\}$ are increasing sequences, and for any $t^n_i \in [T_{k-1},T_k]$, we have $| \by_n(t_i^n)  -  \by(t_i^n)  | \le \Upsilon_{k}  \cdot  \Delta_n^{2\alpha-1}$.
Therefore,  
\$
\sup_{t^n_i \in D_n}| \by_n(t^n_i)  - \by(t^n_i)|& \le \Upsilon_{\lceil (4\zeta)^{1/\alpha} \rceil} \cdot  \Delta_n^{2\alpha-1}. 
\$
Finally, for any $t \in [0,1]$, there exists $i$ such that $t^n_i\le t < t^n_{i+1} $. 
Then we have 
\$
| \by_n(t)- \by(t)   | & \le | \by_n(t^n_i)- \by_n(t)   |+| \by_n(t)- \by(t)   |+| \by(t^n_i)- \by(t)| \\
&\le \Upsilon_{\lceil (4\zeta)^{1/\alpha} \rceil} \cdot  \Delta_n^{2\alpha-1}+G^*_1\cdot \Delta_n^{\alpha} \le \bigl( \Upsilon_{\lceil (4\zeta)^{1/\alpha} \rceil}+G^*_1 \bigr) \cdot  \Delta_n^{2\alpha-1}.
\$ 
\end{proof}

\section{Application to multilevel Monte Carlo} \label{sec:mlmc}
In this section, we demonstrate how our $\epsilon$-strong simulation algorithm can be easily combined with multilevel Monte Carlo (MLMC). We start with a brief introduction of the MLMC framework \citep{giles2008multilevel}. 
Our objective is to estimate $\alpha=\E[g(B^H)]$, where $g$ is a functional of the fBM path. 
The MLMC estimator takes the following form
\begin{equation}\label{eq:mlmc}
\hat\alpha_K=\sum_{k=0}^{K}\frac{1}{r_k}\sum_{i=1}^{r_k}D_k(i)
\end{equation} 
where $D_k(i)$'s are i.i.d. copies of some properly defined level differences.
For example, $D_k(i)=g(B_k^H)-g(B_{k-1}^H)$. Assuming $g(B_{-1}^H)=0$. Then
$\E[\hat\alpha_K]=g(B_K^H)$, which implies the bias the of the estimator only depends on the bias at the
highest level $K$. On the other hand, we have 
$$\V(\hat\alpha_K)=\sum_{k=0}^{K}\frac{1}{r_k}\V(D_k),$$
i.e. the variance depends on the variance at different levels. 
Thus, by using appropriate coupling to create the level differences, $D_k$'s, and smartly allocate
the computational budget, $r_k$'s, we can achieve substantial computational cost reduction
comparing to naive Monte Carlo method \footnote{By naive Monte Carlo method, we mean generating i.i.d. 
copies of $g(B_K^H)$.}.
As we shall explain next, the advantage of our $\epsilon$-strong simulation algorithm is 
that it provides an elegant way to construct $D_k$'s. It is also straightforward to calculate the 
variances of $D_k$'s. In what follows, we denote $\cC(D_k)$ as the computational cost of generating one copy of $D_k$.

We consider two cases for the functional form $g$. \\

\begin{itemize}
\item[I)] $g$ is Lipschitz continuous with respect to the supremum norm, 
and once $B^H$ is given, $g(B^H)$ can be evaluated in closed form. In this case, 
we can construct $D_k(i)$'s as i.i.d.\ copies of $g(B_{k}^H)-g(B_{k-1}^H)$.
The coupling is created by using the same fBM path truncated at different levels for
$B_{k}^H$ and $B_{k-1}^H$. 
\item[II)] $g$ maps $B^H$ to the solution to an SDE at a fixed time point. In this case,
we can set $D_k(i)$'s as i.i.d.\ copies of $Y_{k}-Y_{k-1}$.
Similarly to Case I, the coupling is created by constructing $Y_{k}$ and $Y_{k-1}$ using the same fBM path 
truncated at different levels. \\
\end{itemize}

For Case I, let $L$ denote the Lipschitz constant of $g$. Then for any fixed $\delta\in (0, H)$,
\begin{eqnarray*}
\V(D_k) &\leq& \EE[(g(B_{k}^H)-g(B_{k-1}^H))^2]\\
&\leq&2\cdot \bigl( \EE[(g(B_{k}^H)-g(B^H))^2]+\EE[(g(B_{k-1}^H)-g(B^H))^2]\bigr) \\
&\leq& \Bigl(\frac{2L\rho}{1-2^{-H+\delta}}\Bigr)^2\cdot \Delta_{k}^{2(H-\delta)},
\end{eqnarray*}
where the last inequality follows from the definition of record-breakers.
 As for the computational complexity, according to the analysis in Section \ref{cc}, we have
$\cC(D_k)=O(\Delta_{k}^{-1}\log(\Delta_{k}^{-1}))$,
where the recursive Gaussian bridge based method is used.

In this case, for a given mean square error (MSE) bound $\epsilon^2$, we can 
set $K=C_1\log(1/\epsilon)$, such that $\EE[\hat \alpha_K]-\alpha=O(\Delta_K^{H-\delta})=O(\epsilon)$.
We can also set $r_k=C_2 \Delta_k^{2(H-\delta)}\epsilon^{-2}\log(1/\epsilon)$, such that
$\V(\hat \alpha_K)=O(\epsilon^2)$.
With our choice of $K$ and $r_k$, the total computational cost of $\hat \alpha_K$ is
$$\sum_{k=0}^{K}r_k\cC(D_k)
=O\left(\epsilon^{-2}\log(1/\epsilon)\sum_{k=1}^{K}\Delta_k^{2(H-\delta)-1}\log(\Delta_k^{-1}) \right)\footnote{Here, we treat the cost of generating the last record-breaker as a constant. This is because $N$ does not depend on $\epsilon$ and $N\leq N(\epsilon)$ in most cases}.
$$
When $2(H-\delta)>1$, the cost is $O(\epsilon^{-2}\log(1/\epsilon))$; otherwise the cost is $O(\epsilon^{-1/(H-\delta)}\log(1/\epsilon)^2)$. 
Note that the total computational cost of naive Monte Carlo estimator is $O(\epsilon^{-2-1/(H-\delta)}\log(1/\epsilon))$. 

{We believe that the complexity that can be achieved using MLMC in this setting is near optimal.
This is based on the fact that mid-point displacement decomposition we employ 
achieves the optimal rate of convergence in terms of the bias/error of the truncated fBM path.}\\

For Case II, under Assumption \ref{ass0}, we have
\begin{eqnarray*}
\V(D_k) &\leq& \EE[(Y_{k} - Y_{k-1})^2]\\
&\leq&\EE[(Y_{k} - Y)^2]+E[(Y_{k-1} - Y)^2]\\
&\leq&4G\cdot \Delta_{k}^{2(2\alpha-1)}
\end{eqnarray*}
for any fixed $\alpha\in(1/2,H-\delta)$.
We also have
$\cC(D_k)=O(\Delta_{k}^{-1}\log(\Delta_{k}^{-1}))$.
Following similar lines of analysis as in Case I, we can show that
in Case II, to achieve an MSE of order $\epsilon^2$, we set
$K=C_3\log(1/\epsilon)$ such that $\EE[\hat \alpha_K]-\alpha=O(\Delta_K^{2\alpha-1})=O(\epsilon)$
and $r_k=C_4\Delta_{k}^{2(2\alpha-1)}\epsilon^{-2}\log(1/\epsilon)$.
Then, the computational cost is
$$O\left(\epsilon^{-2}\log(1/\epsilon)\sum_{k=1}^{K}\Delta_k^{2(2\alpha-1)-1}\log(\Delta_k^{-1})\right).$$
When $\alpha>3/4$, the cost is $O(\epsilon^{-2}\log(1/\epsilon))$; otherwise the cost is $O(\epsilon^{-1/(2\alpha-1)}\log(1/\epsilon)^2)$.

 When $\alpha>3/4$, the MLMC achieves the near optimal complexity. 
For the $\alpha\in(1/2,3/4)$, the computational complexity can get arbitrarily bad as $\alpha$ approaches $1/2$.
The fundamental bottleneck here is the Euler scheme. 
If we use higher order discretization scheme like the Milstein scheme,
the convergence rate of the discretization error can be improved from $\Delta_n^{2\alpha-1}$ to $\Delta_n^{3\alpha-1}$.
However, evaluating the iterated integrals of fBM (L\'evy area) is itself a very challenging task.
In this paper, we do not pursue more sophisticated discretization schemes, but we view this as an interesting future research direction.


\section{Numerical experiments} \label{sec:num}
In this section, we conduct some numerical experiments as a sanity check 
of the algorithms we developed.
We also provide some discussions about implementation issues of our algorithm.

Figure \ref{f1} displays two realizations of a fBM with $H=0.8$ using Algorithm \ref{alg:main}
with $\epsilon=0.1$.  

\begin{figure*}
\centering
\begin{tabular}{cc}
\includegraphics[height=0.3\textwidth]{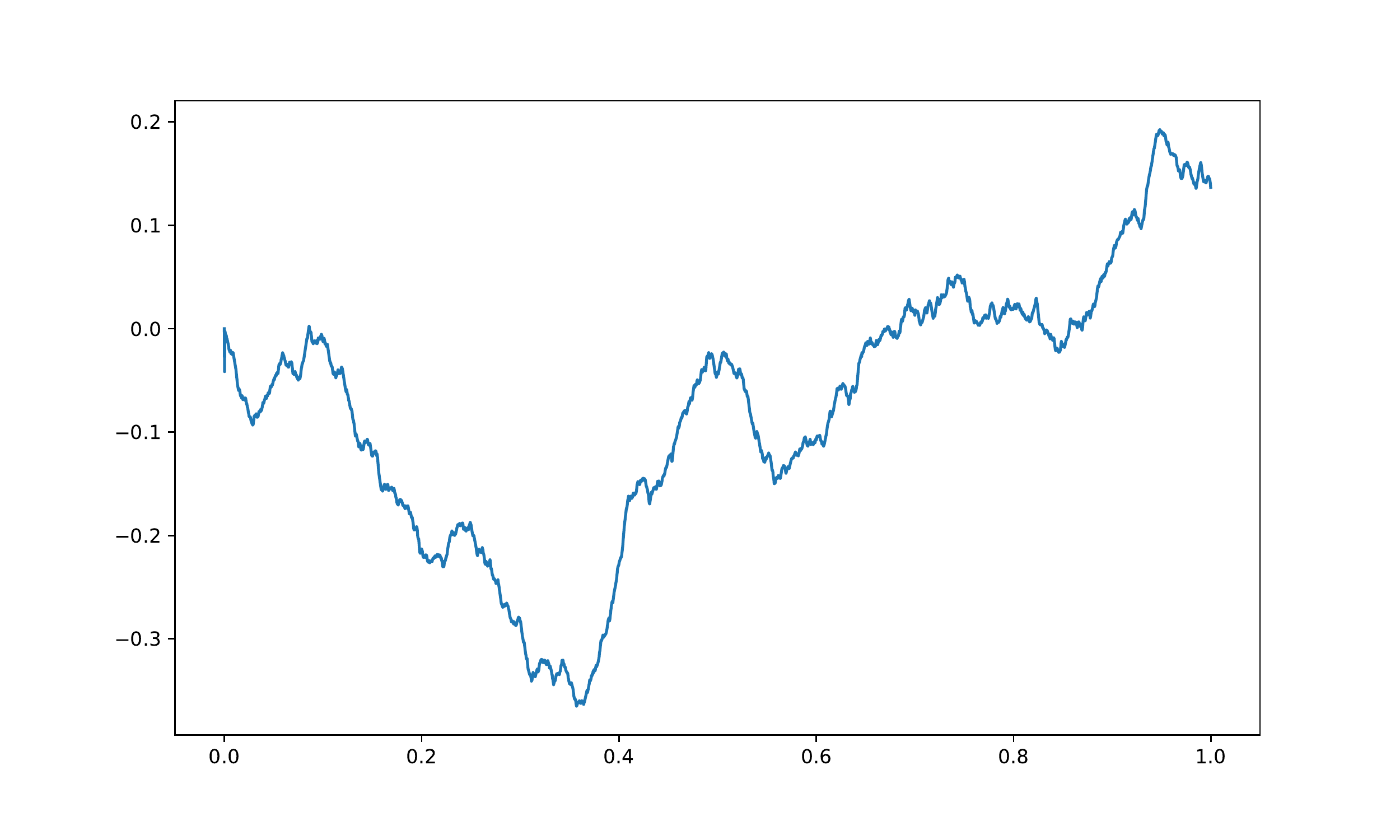}
&
\includegraphics[height=0.3\textwidth]{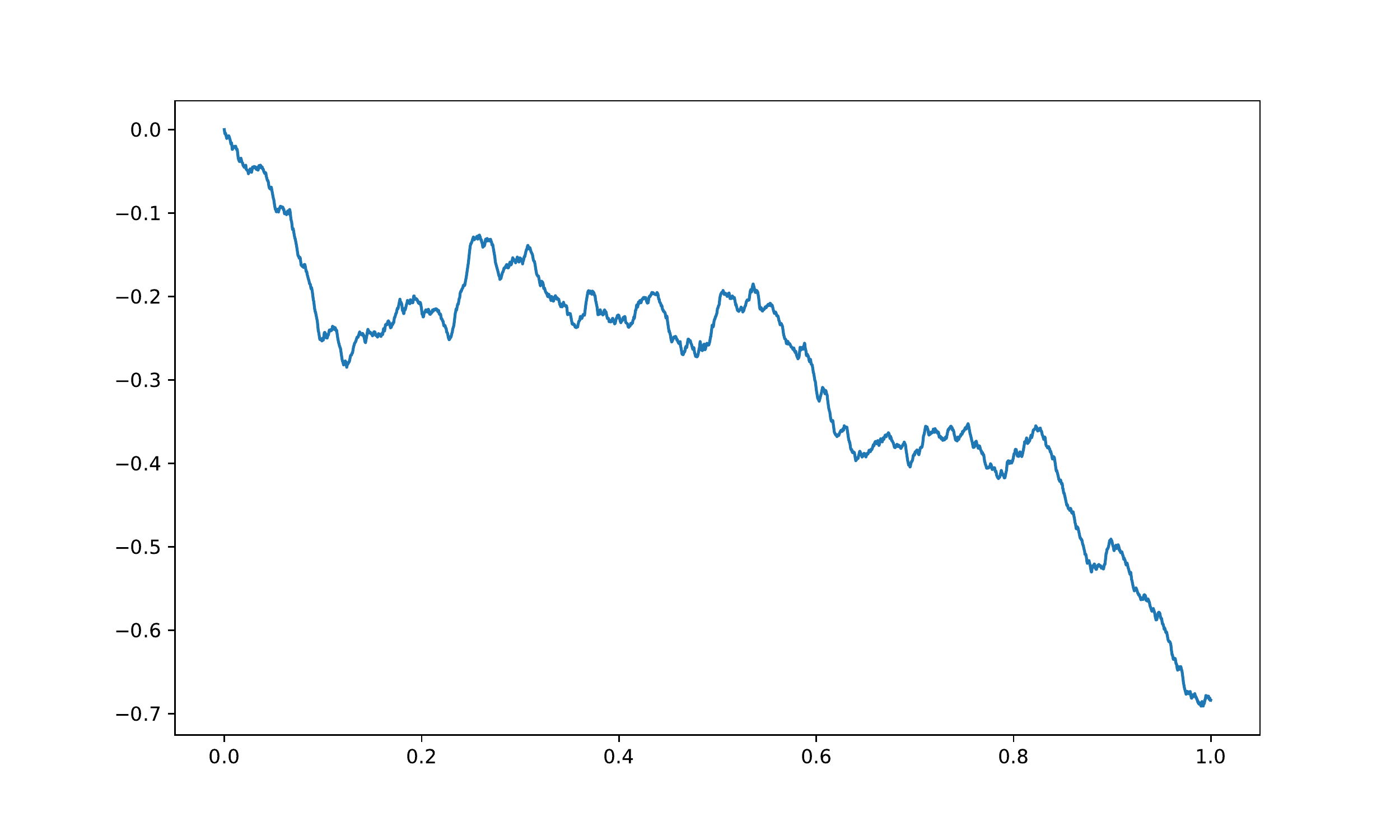}
\\
(a) & (b)\\
\end{tabular}
\caption{\small Two realizations of the $\epsilon$-strong simulation of fBM where the parameters are $H=0.8, \epsilon=0.1,\rho=5,\delta=0.1$. } \label{f1}
\label{fg}
\vspace{-12pt}
\end{figure*}

We next provide a brief discussion about the choices of record-breaker parameters $\rho$ and $\delta$ in practice. 
Recall that the record-breaking threshold takes the form $\rho\Delta_n^{H-\delta}$ for level $n$.
$\delta$ determines the asymptotic decay rate of the bound 
and $\rho$ determines the scale of the bound. Thus, we would want both to be as small as possible
when only considering truncation level, $N(\epsilon)$.
Now, when also taking into account the time of the last record-breaker, $N$,
we notice that larger values of $\rho$ and $\delta$ result in larger thresholds, under which, the records are less likely to be broken. 
These in turn lead to a smaller value of $N$. We also note that larger values of $\rho$ and $\delta$
lead to smaller stating level $N^*(\rho,\delta)$ in Algorithm \ref{SLRB}.
In the asymptotic sense, as $N$ and $N^*(\rho,\delta)$ do not depends on $\epsilon$, 
the values of $N$ and $N^*(\rho,\delta)$ do not matter.
This indicates that in theory, we should set $\rho$ and $\delta$ as small as possible.
However, in practice, we do care about the ``cost" of sampling $N$.
Thus, in actual implementations, we will tune $\rho$ and $\delta$ to 
balance $N^*(\rho,\delta)$ and $N(\epsilon)$.
Table \ref{t1} and \ref{t2} show the truncation level, the starting level, and the average level of the last record-breaker, ($N(\epsilon)$, $N^*(\rho,\delta)$, $\EE[N]$), for different choices of $\rho$ and $\delta$.
Table \ref{t1} is for fBM with $H=0.8$ and Table \ref{t2} is for fBM with $H=0.45$.
We make two observations from the tables.
First, the level of the last record-breaker, $N$, is quite sensitive to our choice of $\rho$. For reasonably large values of $\rho$, e.g. $\rho \geq 2.5$, the record breaker rarely happens beyond level $n=1$.
On the other hand, as we have discussed above, smaller values of $\rho$ lead to smaller values of truncation level $N(\epsilon)$.
Second, the starting level $N^*(\rho,\delta)$ can be unreasonably large if $\rho$ and $\delta$ are not properly chosen. In what follows, we shall 
provide more discussions about $N^*(\rho,\delta)$, including an alternative algorithm to get rid of the starting level requirement. 

\begin{table}[!htbp]
\centering
\begin{tabular}{|c|c|c|c|}
\hline
\diagbox{$\delta$}{$\rho$}&$1$&$2.5$&$5$\\
\hline
$0.1$&$(7,38,14)$&$(9,21,1)$&$(11,1,1)$\\
\hline
$0.2$&$(9,16,6)$&$(11,6,1)$&$(12,1,1)$\\
\hline
\end{tabular}
\caption{\small ($N(\epsilon)$, $N^*(\rho,\delta)$, $\EE[N]$) under different choices of $\rho $ and $\delta $ when $H=0.8, \epsilon=0.1$. }\label{t1}
\end{table}

\begin{table}[!htbp] 
\centering
\begin{tabular}{|c|c|c|c|}
\hline
\diagbox{$\delta$}{$\rho$}&$1$&$2.5$&$5$\\
\hline
$0.1$&$(16,38,15+)$&$(20,21,1)$&$(23,1,1)$\\
\hline
$0.2$&$(24,16,8)$&$(30,6,1)$&$(31,1,1)$\\
\hline
\end{tabular}
\caption{\small ($N(\epsilon)$, $N^*(\rho,\delta)$, $\EE[N]$) under different choices of $\rho $ and $\delta $ when $H=0.45, \epsilon=0.1$. 
($15+$ means even at level $15$ we still see record-breaker happening.)}\label{t2}
\end{table}


Recall from the development of Algorithm \ref{SLRB} that the starting level is required such that the normalizing constant $Z_n$ in \eqref{gm} is smaller than one, and hence, 
the weighted likelihood ratios $ \Theta^+_n(m,k)$ and $ \Theta^-_n(m,k)$ conditional on the proper record-breaking event are also bounded by one. This property (bounded by one) is desirable as we can generate Bernoulli random variable with
probability of success $\EE[ \Theta^+_n(m,k)]$ by generating $1\{U \leq \Theta^+_n(m,k)\}$, where $U$ is a Uniform random variable independent of everything else.
Note that when $Z_n>1$, with the change-of-measure technique we used in Algorithm \ref{ECM}, we are only able to generate Bernoulli random variables with success probability $\PP_n(\tau=n+m)/Z_n$. But our initial objective is to generate Bernoulli random variables with success probability $\PP_n(\tau=n+m)$. We can use a technique called the {\em Bernoulli factory} to overcome this gap. We next introduce the basic idea of Bernoulli factory and explain how it applies to our setting.
The main objective of introducing this alternative to Algorithm \ref{ECM} is to get
rid of the starting level requirement in Algorithm \ref{SLRB}. 

Suppose that $X_1,X_2, \cdots $ are i.i.d. Bernoulli random variables with unknown success probability $p$. Given a known function $f$, a Bernoulli factory takes $X_1,X_2, \cdots $ as input and outputs a Bernoulli random variable with success probability $f(p)$. 
In our case, by sampling the fBM under the change of measure up to level $(n+m)$ and check whether the record is broken at level $n+m$, we are able to ``generate" a Bernoulli random variable with success probability $\PP_n(\tau=n+m)/Z_n$. Then our $f(p)$ function is a linear function in $p$, i.e. $f(p)=Z_np$. This is also known as a linear Bernoulli factory. We refer to \cite{huber2016nearly} for a nearly optimal linear Bernoulli factory that can be directly applied to our setting. 

To sum up the discussion here, as record-breaking is a rare event \footnote{The corresponding probability decays double exponentially fast in $n$.}, the benefits of shrinking the truncation level by choosing small $\rho$ and $\delta$ is very appealing.
However, in practice, small values of $\rho$ and $\delta$ may lead to a large value of $N^*(\rho,\delta)$.
When $N^*(\rho,\delta)$ is impractically large, we may consider using the Bernoulli factory to get rid of the starting level requirement. 
Even though, in theory, our algorithm achieves near optimal complexity for fBM, 
there are still a lot of rooms for improvements in practical implementations.
Moreover, as we have discussed in Section \ref{sec:mlmc}, for fBM driven SDEs, even though our algorithm achieves near optimal complexity
under the Euler scheme, there are a lot of rooms for improvements if higher order 
discretization schemes can be efficiently implemented. We view all these as interesting future research directions.


\appendix
\section{Detailed Proofs for Technical Lemmas in Section \ref{sec:mid} } \label{AA}
\subsection{Proof of Lemma \ref{var_bound}}
According to the definitions of $a^k_j,c^k_j$ and $\bB^H_k$,  we have 
\$
\VV_{k-1}(a^k_j-c^k_j )=\VV\bigl(B^H(t^{k}_{2j+1}) | \bB^H_{k-1} \bigr)&=\VV\bigl(B^H(t^{k}_{2j+1}) | B^H(t^{k-1}_0), B^H(t^{k-1}_1),\cdots,B^H(t^{k-1}_{2^{k-1}})\bigr) \\
&\le \VV\bigl(B^H(t^{k}_{2j+1}) | B^H(t^{k-1}_j)\bigr).
\$
where the inequality follows from the fact that for random variables $X,Y$ and $Z$, we have 
$\EE[\V(X|Y,Z)] \le \EE[\V(X|Y)]$ and when the joint distribution of $(X,Y,Z)$ is multivariate Gaussian, the conditional variance $\V(X|Y,Z) $ and $\V(X|Y)$ are both constants.

The upper bound of conditional variance $\VV(B^H(t^{k}_{2j+1}) | B^H(t^{k-1}_j)) $ is based on the orthogonal bridge decomposition of Gaussian bridge (\cite{sottinen2014generalized}). Specifically, let $r(\cdot,\cdot)$ denote the covariance function of $ B^H$, then the distribution of $B^H(t)$ conditional on $B^H(t^{k-1}_j)=y^* $ is same with that of 
\$ 
B^H_{t^{k-1}_j,y^*}(t)=B^H(t)-\frac{r(t,t^{k-1}_j)}{r(t^{k-1}_j,t^{k-1}_j)}\cdot  (B^H(t^{k-1}_j)-y^*). 
\$
By simple calculation, the covariance function of $ B^H_{t^{k-1}_j,y^*}(t) $ is given by
\$
r_{t^{k-1}_j,y^*}(s,t)=r(s,t)-\frac{r(t,t^{k-1}_j)\cdot r(t^{k-1}_j,s)}{r(t^{k-1}_j,t^{k-1}_j)}.
\$ 
Recall that $ \Delta_{k}=t^{k}_{2j+1}-t^{k-1}_j=2^{-k}$, then we have 
\$
\VV\bigl(B^H(t^{k}_{2j+1}) | B^H(t^{k-1}_j)\bigr)&=r_{t^{k-1}_j,y^*}\bigl(t^{k}_{2j+1},t^{k}_{2j+1}\bigr)\\
&=|t^{k-1}_j+\Delta_{k}|^{2H} -1/4\cdot \bigl(|t^{k-1}_j+\Delta_{k}|^{2H}+|t^{k-1}_j|^{2H}-|\Delta_{k}|^{2H}\bigr)^2\cdot |t^{k-1}_j|^{-2H}.
\$ 
In the next, we show that $\VV(B^H(t^{k}_{2j+1}) | B^H(t^{k-1}_j))$ is no greater than  $ 2\cdot \Delta_{k}^{2H}   $  , 
which is equivalent to  
\$
|2j+1|^{2H}-1/4\cdot \bigl(|2j+1|^{2H}+|2j|^{2H}-1\bigr)^2\cdot |2j|^{-2H}\le 2,
\$
for arbitrary positive integers $j$. By calculation, the left-hand side of above inequality becomes
\$
\frac{1}{4}\cdot \biggl[2\cdot |2j+1|^{2H}-\Bigl(\frac{|2j+1|^{4H}}{|2j|^{2H}}+|2j|^{2H}\Bigr)\biggr]-\frac{1}{4}\cdot|2j|^{-2H}+\frac{1}{2}\cdot\Bigl(\frac{|2j+1|^{2H}}{|2j|^{2H}}+1\Bigr).
\$
Obviously, this quantity is no greater than $2$. Thus, above all, we obtain
\$
\VV_{k-1}(a^k_j-c^k_j )\le \VV\bigl(B^H(t^{k}_{2j+1}) | B^H(t^{k-1}_j)\bigr) \le 2\cdot \Delta_{k}^{2H} =2\cdot 2^{-2kH},
\$
which concludes the proof.

\subsection{Proof of Lemma \ref{exp_bound}}
Let $Y=\max_{1\le i \le n}|X_i|$. For any $t>0$, by Jensen's inequality, we have 
\$
\exp \bigl \{ t\EE[Y] \bigr \}&\le \EE\bigl[\exp\{tY\}\bigr]=\EE\bigl[\max_{1\le i \le n}\exp\{t|X_i|  \}\bigr]\le \sum_{i=1}^n\EE\bigl[\exp\{t|X_i|  \}\bigr] \\
&\le  \sum_{i=1}^n\EE\bigl[\exp\{tX_i  \}\bigr]+\EE\bigl[\exp\{-tX_i  \}\bigr] \le 2n\cdot \exp \{t^2\sigma^2/2   \}.
\$
Hence, we obtain
$\EE[Y]\le \log(2n)/t+t\sigma^2/2$ and the result follows by setting $t=\sqrt{2\log(2n)}/\sigma$.

\subsection{Proof of Lemma \ref{var_bound2}}
In the following, we denote by
\$
\bSigma_{11 }^{(k-1)}=
\left[
 \begin{matrix}
  1^{2H} & (1^{2H}+3^{2H}-2^{2H})/2 & (1^{2H}+5^{2H}-4^{2H})/2 & \cdots & \cdots \\
   (3^{2H}+1^{2H}-2^{2H})/2 & 3^{2H} & (3^{2H}+5^{2H}-2^{2H})/2 & \cdots & \cdots  \\
   (5^{2H}+1^{2H}-4^{2H})/2 & (5^{2H}+3^{2H}-2^{2H})/2 & 5^{2H} & \cdots & \cdots  \\
   \cdots & \cdots & \cdots & \cdots & \cdots    \\
   \cdots & \cdots & \cdots & \cdots & \cdots  \\
  \end{matrix}
  \right]_{2^{k-1} \times 2^{k-1}},
\$
which is a $2^{k-1}$  by $2^{k-1}$ matrix with $(i,j)$-th entry $ (|2i-1|^{2H}+ |2j-1|^{2H}-|2i-2j|^{2H})/2$. Recall the definition of $\bN_{k-1}$ and $\bM_{k-1}$, we have 
\# \label{matrix_dec}
&(\bN_{k-1}-\bM_{k-1}) \bSigma_{22}^{({k-1})}  (\bN_{k-1}-\bM_{k-1})^{\top} \notag \\ 
\quad &=\bM_{k-1} \bSigma_{22}^{({k-1})}\bM_{k-1}^{\top}-\bM_{k-1}[\bSigma_{12}^{({k-1})}]^{\top}-\bSigma_{12}^{({k-1})}\bM_{k-1}^{\top}+\bSigma_{12}^{({k-1})}\cdot [\bSigma_{22}^{({k-1})}]^{-1}[\bSigma_{12}^{({k-1})}]^{\top}  \notag \\
\quad &=\bigl(\bM_{k-1} \bSigma_{22}^{({k-1})}\bM_{k-1}^{\top}-\bM_{k-1}[\bSigma_{12}^{({k-1})}]^{\top}-\bSigma_{12}^{({k-1})}\bM_{k-1}^{\top}+\bSigma^{({k-1})}_{11}\bigr) \notag \\
 &\quad \quad  -\bigl(\bSigma^{({k-1})}_{11}  -\bSigma_{12}^{({k-1})}\cdot [\bSigma_{22}^{({k-1})}]^{-1}[\bSigma_{12}^{({k-1})}]^{\top}\bigr).
\#
Note that $(\bSigma^{({k-1})}_{11}
-\bSigma_{12}^{({k-1})}\cdot [\bSigma_{22}^{({k-1})}]^{-1}[\bSigma_{12}^{({k-1})}]^{\top})\cdot \Delta_{k}^{2H}$ is the conditional covariance matrix 
\$\mathbb{C}ov\Bigl[\bigl(B^H(t^{k}_{1}),B^H(t^{k}_{3}),\cdots,B^H(t^{k}_{2^{k}-1})\bigr) \big| \bB^H_{k-1} \Bigr]. \$  
Hence, the diagonal entries in second part in \eqref{matrix_dec} are nonnegative. In the next, we compute the diagonal entries of the first part in \eqref{matrix_dec}. We use $\xi_{i,j}$ and $\eta_{i,j}$ to denote the $(i,j)$-th entry of $\bSigma_{22}^{({k-1})}$ and $\bSigma_{12}^{({k-1})}$, respectively. After some calculation, we obtain that the $j$-th diagonal entries of $\bM_{k-1} \bSigma_{22}^{({k-1})}\bM_{k-1}^{\top}$ and $\bM_{k-1}[\bSigma_{12}^{({k-1})}]^{\top}  $  are 
\$
\bigl[\bM_{k-1} \bSigma_{22}^{({k-1})}\bM_{k-1}^{\top}\bigr]_{j,j}&=1/4\cdot (\xi_{j,j}+\xi_{j,j+1}+\xi_{j+1,j}+\xi_{j+1,j+1}), \\
\bigl[\bM_{k-1}[\bSigma_{12}^{({k-1})}]^{\top}\bigr]_{j,j}&=1/2\cdot (\eta_{j,j}+\eta_{j+1,j}).
\$ 
By plugging in expression of $\xi_{i,j}, \eta_{i,j}$ and $ \bSigma^{({k-1})}_{11}  $,  we obtain that the $j$-th diagonal entry of the first part in \eqref{matrix_dec}  is $ 1-2^{2H-2}  $, which is a constant.  As a result, the $j$-th diagonal entry of $ \bSigma^{({k})} $ is upper bounded by 
\$
[\bSigma^{(k)}]_{j,j} \le (1-2^{2H-2})\cdot \Delta_{k}^{2H}< 2\cdot 2^{-2kH}.
\$ 

\subsection{Proof of Lemma \ref{holer_bound}}
We use $f_k(t)$ to denote $B^H_k(t)-B^H_{k-1}(t)$ and then by definition, 
\$
\bigl\|B^H_k-B^H_{k-1}\bigr\|_{\alpha}=\|f_k\|_{\alpha}=\sup_{0\le s<t \le 1}  \frac{|f_k(s)-f_k(t)|}{|s-t|^{\alpha}}.
\$
Recall that $f_k(t^{k-1}_j)=0, f_k(t^{k}_{2j+1})=a^k_j-b^k_j$ and $f_k(t)$ is linear over intervals $[t^{k}_{2j},t^{k}_{2j+1} ]$ and $[t^{k}_{2j+1},t^{k}_{2j+2}]$, where $j=0,\cdots,2^{k-1}-1$. Note that $t^{k}_{2j}=t^{k-1}_j$.  Let \$\kappa= 2^{k}\cdot \max_{0\le j \le 2^{k-1}-1} |a^k_j-b^k_j| .\$ Then $\kappa$ is the maximal slope of all linear pieces of $f_k(t)$.    We make a  discussion based on the locations of $s$ and $t$.
\vskip2pt
\noindent \textbf{case 1:  $|s-t|\le 2^{-(k-1)}$. } If there exists some $j^*$ such that $s,t \in [t^{k}_{2j^*},t^{k}_{2j^*+2}] ,$  since $\kappa$ is the maximal slope, it is easy to show that $|f_k(s)-f_k(t)|\le \kappa \cdot |s-t| $. Otherwise, there exists some $j^*$ such that $  t^{k}_{2j^*-1}\le s<t^{k}_{2j^*}<t\le t^{k}_{2j^*+1}$. Then we have 
\$
|f_k(s)-f_k(t)|&=|f_k(s)-f_k(t^{k}_{2j^*})+f_k(t^{k}_{2j^*})-f_k(t)|\le |f_k(s)-f_k(t^{k}_{2j^*})|+|f_k(t^{k}_{2j^*})-f_k(t)|\\
&\le 2\kappa\cdot\bigl(  |s-t^{k}_{2j^*} |+|t- t^{k}_{2j^*}  |\bigr)=2\kappa\cdot |s-t| .
\$ 
 Hence, by definition, we obtain 
\$
 \frac{|f_k(s)-f_k(t)|}{|s-t|^{\alpha}} \le 2\kappa \cdot |s-t|^{1-\alpha} \le 2\kappa \cdot 2^{-(1-\alpha)k}=2^{\alpha(k-1)+2}\cdot \max_{0\le j \le 2^{k-1}-1} |a^k_j-b^k_j|.
\$   

 \noindent \textbf{case 2: $|s-t|> 2^{-(k-1)}$.} In this case, there exist some $i<j$ such that $s \in [t^k_{i},t^k_{i+1}]  $ and $t \in [t^k_{j},t^k_{j+1}]  $. Then we have 
\$
|f_k(s)-f_k(t)|&=|f_k(s)-f_k(t^k_{i+1})+f_k(t^k_{j})-f_k(t)|\le |f_k(s)-f_k(t^k_{i+1})|+|f_k(t^k_{j})-f_k(t)| \\
&\le 2\kappa\cdot\bigl(  |s-t^k_{i+1} |+|t- t^k_{j}  |\bigr) \le 2^{-k+2}\cdot \kappa.
\$ 
 Moreover, we have 
\$
 \frac{|f_k(s)-f_k(t)|}{|s-t|^{\alpha}} \le \frac{2^{-k+2}\cdot \kappa}{2^{-\alpha (k-1)}}=2^{\alpha (k-1)+2}\cdot \max_{0\le j \le 2^{k-1}-1} |a^k_j-b^k_j|.
\$
 Above all, we obtain $ \|B^H_k-B^H_{k-1}\|_{\alpha}\le 2^{\alpha (k-1)+2}\cdot \max_{0\le j \le 2^{k-1}-1} |a^k_j-b^k_j|,  $ which concludes the proof of Lemma \ref{holer_bound}. 

\section{Detailed Proofs for Technical Lemmas in Section \ref{sec:alg} } \label{AB}
\subsection{Proof of Lemma \ref{lrb} }
Before we prove Lemma \ref{lrb},
we first establish an upper bound on $\Xi_{n}^{(m,k)}$, i.e. the conditional moment generating function of $ \bbeta^{\top} \balpha_n(m,k)$. 
\begin{lemma} \label{cor}
Under the BCE condition, for $m\ge 1$, $1\le k \le 2^{n+m-1}$ and $\theta>0$,
\$
\Xi_{n}^{(m,k)}(\theta) \le \exp\Bigl\{1/2\cdot\bigl( \rho \ell_{{\tau_n}+m}\cdot \theta + \theta^2\cdot 2^{-2({\tau_n}+m)H} \bigr)  \Bigr\}.
\$
\end{lemma}

\begin{proof}
We denote by
\$
\mu_{n}(m,k)=\E\bigl[ \bbeta^{\top} \balpha_n(m,k) \big| \bB^H_{n}\bigr],\ \sigma^2_{n}(m,k)=\V\bigl( \bbeta^{\top} \balpha_n(m,k) \big|\bB^H_{n}\bigr).
\$ 
The key of the proof is to establish proper upper bounds for $\mu_{n}(m,k)$ and $\sigma^2_{n}(m,k)$.  
We will first show that
\$
\sigma^2_{n}(m,k) \le 2^{-2(n+m)H}.
\$ 
Similar to the proof of Lemma \ref{var_bound}, for Gaussian random vectors, taking condition reduces the variance. Hence, we have 
\$
& \V\Bigl( \bbeta^{\top} \balpha_n(m,k) \big| \bB^H_{n}\Bigr) 
\le \V \Bigl( 1/2\cdot \bigl(B^H(t^{{n}+m}_{2k-2})+B^H(t^{{n}+m}_{2k})\bigr)-B^H(t^{{n}+m}_{2k-1}) \Bigr)\\
&\le 2\cdot \V \Bigl( 1/2\cdot  \bigl(B^H(t^{{n}+m}_{2k-2})-B^H(t^{{n}+m}_{2k-1}\bigr)    \Bigr) +2 \cdot \V \Bigl( 1/2\cdot  \bigl(B^H(t^{{n}+m}_{2k})-B^H(t^{{n}+m}_{2k-1}\bigr)  \Bigr).
\$
Since the fBM is a stationary process, we have 
\$
\V \bigl( B^H(t^{{n}+m}_{2k-1})-B^H(t^{{n}+m}_{2k-2})  \bigr)=\V \bigl( B^H(t^{{n}+m}_{2k})-B^H((t^{{n}+m}_{2k-1}  \bigr)=\V \bigl( B^H(\Delta_{{n}+m}) \bigr)=\Delta_{{n}+m}^{2H}, 
\$
which further implies that 
\$ 
 \sigma^2_{n}(m,k) \le 2^{-2({n}+m)H}. 
\$
Under the BCE condition, we have $ \mu_n(m,k)\le \rho/2\cdot \ell_{n+m}$. Then the upper bound of the moment generating function $\Xi_{n}^{(m,k)}(\theta)  $ holds.
\end{proof}

We turn to the proof of Lemma \ref{lrb} now. We prove the case when $\pi=+$ only. The case when $\pi=-$ follows analogously.
Note that if there is a up-crossing record-breaker at level ${n}+m$ position $k$, then $ \bbeta^{\top} \balpha_n(m,k)> \rho \ell_{{n}+m}  $. In this case, we have 
\$
\Theta_{n}^+(m,k)&=g_{n}(m)^{-1}\cdot 2^{{n}+m}\cdot \exp\Bigl\{ -\theta_{n}^+(m)\cdot  \bbeta^{\top} \balpha_n(m,k)+\log\bigl(\Xi_{n}^{(m,k)}(\theta^+_{n}(m))\bigr)   \Bigr \} \\
& \le g_{n}(m)^{-1}\cdot 2^{{n}+m}\cdot \exp\{ -\theta_{n}^+(m)\cdot \rho \ell_{{n}+m}\}\cdot \exp\Bigl\{1/2\cdot\bigl( \rho \theta_{n}^+(m)  \ell_{{n}+m}+ \theta_{n}^+(m)^2\cdot 2^{-2({n}+m)H} \bigr)  \Bigr\} \\
& \le g_{n}(m)^{-1}\cdot 2^{{n}+m}\cdot \exp\Bigl\{ -\rho/2\cdot \theta_{n}^+(m)\cdot 2^{-({n}+m)(H-\delta)}+1/2\cdot  \theta_{n}^+(m)^2\cdot 2^{-2({n}+m)H}   \Bigr \},
\$ 
where the second inequality follows from Lemma \ref{cor}. Hence, by our choice of \$\theta_{n}^+(m)=\rho/2\cdot 2^{({n}+m)(H+\delta)},\$ we obtain  
\$
\Theta_{n}^+(m,k)\le g_{n}(m)^{-1}\cdot 2^{{n}+m}\cdot \exp\bigl\{-\rho^2/8\cdot 2^{2({n}+m)\delta} \bigr\}=Z_n \le 1,
\$ 
for all $n>N^*(\rho,\delta)$.

\subsection{Proof of Lemma \ref{finitebce}}
We first consider the conditional expectation. Note that 
\$
\mu_n(m,k)=\E\bigl[ \bbeta^{\top} \balpha_n(m,k)\big| \bB^H_{n}\bigr]=(1/2,-1,1/2)^{\top}\bSigma_{n}^{(m,k)}\bSigma_{n}^{-1}\bB^H_{n}.
\$
Since $ \mu_n(m,k) $ is the linear combination of Gaussian random variables, itself is also Gaussian.
It is easy to see that $\EE[\mu_n(m,k)]=0$.   
For the variance, according to the decomposition of conditional variance, we have 
\$
\VV\bigl( \bbeta^{\top} \balpha_n(m,k)\bigr)=\EE\bigl[\VV( \bbeta^{\top} \balpha_n(m,k)|\bB^H_n)\bigr] +\VV\bigl( \EE[ \bbeta^{\top} \balpha_n(m,k)|\bB^H_n]\bigr).
\$
Thus we obtain
\$
\VV(\mu_n(m,k))&\le \VV\bigl( \bbeta^{\top} \balpha_n(m,k) )\bigr)+\EE\bigl[\VV( \bbeta^{\top} \balpha_n(m,k)|\bB^H_n)\bigr]\le 2\cdot 2^{-2(n+m)H},
\$
where the inequality follows from the proof of Lemma \ref{cor}. For fixed $n,m,k$, we define the event
\$
\cE_{n}(m,k)=\{  | \mu_{n}(m,k) | > \rho/2 \cdot 2^{-(n+m)(H-\delta)}\}.
\$
Then we have 
\$
\PP\bigl(\cE_{n}(m,k)\bigr)\le C \exp\{ -\rho^2/8\cdot 2^{2(n+m)\delta} \},
\$
where $C$ is some constant. Note that  
$
\cE_n\subseteq \cup_{m=1}^{\infty} \cup_{k=1}^{2^{n+m-1}} \cE_n(m,k).
$
Then we have 
\$
\PP(\cE_n)& \le \sum_{m=1}^{\infty} 2^{n+m-1} \cdot \PP\bigl(\cE_n(m,k)\bigr ) \\
&\le C\cdot \sum_{m=1}^{\infty} 2^{n+m-1} \cdot \exp\{ -\rho^2/8\cdot 2^{2(n+m)\delta} \}\le C^{\prime }\cdot \exp\{ -C''\cdot 2^{2(n+m)\delta} \},
\$
where $C'$ and $C''$ are some constants. Since 
\$
\sum_{n=1}^{\infty}  C^{\prime }\cdot \exp\{ -C''\cdot 2^{2(n+m)\delta} \}< \infty,
\$
by Borel-Cantelli Lemma, events $\{\cE_n\}_{n\ge1 }$ happen finite times almost surely. Moreover, similar to the proof of Theorem \ref{ucr}, we further have that the moment generating function of $\cE$, the last level where the BCE condition is broken, exists everywhere.

\subsection{Proof of Lemma \ref{ceb}}
Recall that we have 
\$
\mu_n(m,k)=\E\bigl[ \bbeta^{\top} \balpha_n(m,k) \big| \bB^H_{n}\bigr]=(1/2,-1,1/2)^{\top}\bSigma_{n}^{(m,k)}\bSigma_{n}^{-1}\bB^H_{n}.
\$
Here $\bSigma_{n}$ is the covariance matrix of  $\bB^H_{n}$ and  $\bSigma_{n}^{(m,k)}$ is the covariance matrix of  $\balpha_{n}(m,k)$ and $\bB^H_{n}$, which has form
\$
\bSigma_{n}^{(m,k)}=
\left[
 \begin{matrix}
  r\bigl((2k-2)/2^{{n}+m},0\bigr) &  r\bigl((2k-2)/2^{{n}+m},1/2^{n}\bigr) & \cdots & \cdots & r\bigl((2k-2)/2^{{n}+m},1\bigr) \\
   r\bigl((2k-1)/2^{{n}+m},0\bigr) & r\bigl((2k-1)/2^{{n}+m},1/2^{n}\bigr) & \cdots & \cdots & r\bigl((2k-1)/2^{{n}+m},1\bigr) \\
   r\bigl(2k/2^{{n}+m},0\bigr) & r\bigl(2k/2^{{n}+m},1/2^{n}\bigr) & \cdots & \cdots & r\bigl(2k/2^{{n}+m},1\bigr)  \\
  \end{matrix}
  \right]_{3 \times (2^{n}+1)},
\$
where $r(\cdot,\cdot)$ denotes the covariance function of fBM. 
By calculation, for $1\le j \le 2^{n}+1$, the $j$-th entry of $ (1/2,-1,1/2)^{\top}\bSigma_{n}^{(m,k)}$ is 
\$
\frac{\Delta_{m+{n}}^{2H}}{4}&\cdot \Bigl( |2k-2|^{2H}+|2k|^{2H}-2|2k-1|^{2H} \\
&+ |2k-2-j 2^{m}|^{2H}+|2k-j 2^{m}|^{2H}-2|2k-1-j 2^{m}|^{2H}  \Bigr).
\$
We consider the function $\phi(x)=(x+2)^{2H}+x^{2H}-2\cdot(x+1)^{2H}$. Then we have $\phi^{\prime}(x)=2H\cdot((x+2)^{2H-1}+x^{2H-1}-2\cdot(x+1)^{2H-1})$. When $1/2<H<1$, by Jensen's inequality, $\phi^{\prime}(x)\le 0$ for $x\ge0$. Hence $\phi(x)$ is decreasing on $[0,\infty]$. Moreover, $\phi(x)$ is convex on $[0,\infty]$. As a result, for all positive integer $i$, we have
\$
0\le (i+2)^{2H}+i^{2H}-2\cdot(i+1)^{2H}\le 2^{2H}-2 <2.
\$ 
When $0<H<1/2$, by Jensen's inequality, $\phi^{\prime}(x)\ge 0$ for $x\ge0$. Hence $\phi(x)$ is increasing on $[0,\infty]$. Moreover, $\phi(x)$ is concave on $[0,\infty]$. As a result, for all positive integer $i$, we have
\$
-2< 2^{2H}-2 \le (i+2)^{2H}+i^{2H}-2\cdot(i+1)^{2H}\le 0.
\$ 
Hence, for all $H$ and $i$, we have $ ||2i-2|^{2H}+|2i|^{2H}-2|2i-1|^{2H}|\le 2$, which implies that the absolute value of any entry of $ (1/2,-1,1/2)^{\top}\bSigma_{n}^{(m,k)}$ is bounded by $\Delta_{{n}+m}^{2H}=2^{-2({n}+m)H}$. 
Recall that $ \gamma_{n} $ denotes the maximal absolute value of the entries in vector $ \bSigma_{n}^{-1}\bB^H_{n} $. Hence,  we have
\$
|\mu_{n}(m,k)|=\bigl| (1/2,-1,1/2)^{\top}\bSigma_{n}^{(m,k)}\bSigma_{n}^{-1}\bB^H_{n} \bigr|\le \gamma_{n}\cdot (2^{n}+1)\cdot 2^{-2({n}+m)H}.
\$

\section{Detailed Proofs for Technical Lemmas in Section \ref{sec:sde} } \label{AC}
\subsection{Proof of Lemma \ref{lem0.0}}
Let $\omega=(2dhC_{\alpha}K(2\alpha)|\nabla \bbf|)^{-1/\alpha}$ and $T_k=k\omega,\ k=0,1,2,\cdots$. Then the union of $ [T_k, T_{k+1}]$ where $ {k=0,\cdots, \lceil \omega^{-1}\rceil},$ covers $[0,1]$.
According to the property of Young integral, we have that for all $s,t \in [T_k,T_{k+1}]$, 
\$
\Bigl| \int_s^t \bbf(\by(u))\ud \bx(u) \Bigr| &\le \bigl| \bbf(\by(s))\cdot (\bx(t)-\bx(s)) \bigr| + d\cdot K(2\alpha)\cdot H_{\alpha,[T_k,T_{k+1}]}(\bbf(\by))\cdot C_{\alpha}\cdot |s-t|^{\alpha} \\ 
&\le \Bigl( h\cdot |\bbf|\cdot C_{\alpha}+dh\cdot K(2\alpha)\cdot C_{\alpha}\cdot |\nabla \bbf| \cdot H_{\alpha,[T_k,T_{k+1}]}(\by)\cdot |s-t|^{\alpha} \Bigr)\cdot |s-t|^{\alpha}.
\$
Since $ \by $ is the solution of equation \eqref{ode}, it is easy to see that 
\$
H_{\alpha,[T_k,T_{k+1}]}(\by) \le h\cdot |\bbf|\cdot C_{\alpha}+dh\cdot K(2\alpha)\cdot C_{\alpha}\cdot |\nabla \bbf| \cdot H_{\alpha,[T_k,T_{k+1}]}(\by)\cdot \omega^{\alpha}, 
\$ 
which further implies that 
\$
\sup_{0\le k \le \lfloor \omega^{-1} \rfloor}H_{\alpha,[T_k,T_{k+1}]}(\by) \le 2h\cdot |\bbf|\cdot C_{\alpha}.
\$
Now we turn to bound $H_{\alpha,[0,1]}(\by).$ For any $s, t\in [0,1]$, we assume that $ T_i \le s < T_{i+1}\le \cdots  T_{j} \le t <T_{j+1}$. Then we have 
\$
|\by(s)-\by(t)|&\le |\by(s)-\by(T_{i+1})|+\sum_{\ell=i+1}^{j-1} |\by(T_{\ell})-\by(T_{\ell+1})|+|\by(T_{j})-\by(t)|\\
&\le \sup_{0\le k \le \lfloor \omega^{-1} \rfloor}H_{\alpha,[T_k,T_{k+1}]}(\by) \cdot \Bigl( |s-T_{i+1}|^{\alpha}+\sum_{\ell=i+1}^{j-1} | T_{\ell}-T_{\ell+1} |^{\alpha}+ |T_{j}-t|^{\alpha}  \Bigr).
\$
Since $\alpha<1$, using Jensen's inequality and the number of covering subintervals, we have 
\$
 |s-T_{i+1}|^{\alpha}+\sum_{\ell=i+1}^{j-1} | T_{\ell}-T_{\ell+1} |^{\alpha}+ |T_{j}-t|^{\alpha}\le  \lceil \omega^{-1}\rceil^{1-\alpha}\cdot |s-t|^{\alpha}.
\$
So we have  $H_{\alpha,[0,1]}(\by)\le 2h \cdot  \lceil \omega^{-1}\rceil^{1-\alpha}\cdot |\bbf|\cdot C_{\alpha}$, which is the first bound.

For another bound, note that according to Young-L\'oeve estimate, for any partition $\Pi$ of $[s,t]$, we have 
\$
&\Bigl| \sum_{t_i\in \Pi}\bbf(\by(t_i))\cdot(\bx(t_{i+1})-\bx(t_{i}))-\bbf(\by(s))\cdot(\bx(t)-\bx(s)) \Bigr| \le dh\cdot K(2\alpha)\cdot |\nabla \bbf| \cdot H_{\alpha,[0,1]}(\by)\cdot  C_{\alpha}|s-t|^{2\alpha}.
\$
Let the partition mesh goes to zero, we obtain the result.

\subsection{Proof of Lemma \ref{lem0.1}}
For convenience, in this proof, we use $\bx_k$ to denote $\bx(t^n_k)$ and $\by_k$ to denote $\by_n(t^n_k)$. For each $0\le j\le \ell\le 2^n$, let $ \bI_{j\ell}=\by_{\ell}-\by_{j}-\bbf(\by_j)\cdot (\bx_{j+1}-\bx_{j}). $
 We first show that for all $t^n_{r},t^n_{j} \in [0,1] $, whenever $|t^n_{r}-t^n_{j}|\le \omega  $, then  $|\bI_{jr}|\le L|t^n_{r}-t^n_{j}|^{2\alpha}$. 
By the definition of Euler scheme,  $\bI_{jr}=0$, if $r-j=0,1$. For $r-j\ge 2$, we prove this lemma via induction. Suppose that the claim holds true for all pairs $p,q$ with $q-p<r-j$. Let $ \ell \in [j,r)$ be the largest integer such that $ |t^n_{j}-t^n_{\ell}|\le 1/2\cdot|t^n_{j}-t^n_{r}|$. Then we have $ |t^n_{\ell+1}-t^n_{r}|\le 1/2\cdot|t^n_{j}-t^n_{r}|$. By the inductive hypothesis, $|\bI_{j\ell}|\le L|t^n_{\ell}-t^n_{j}|^{2\alpha}$ and hence, we have  
\$
|y^i_{\ell}-y^i_{j}| \le |I^{i}_{j\ell}|+\Bigl|\sum_{k=1}^hf_{ik}(\by_j)(x^k_{\ell}-x^k_{j})  \Bigr | \le L|t^n_{\ell}-t^n_{j}|^{2\alpha}+h\cdot |\bbf|\cdot C_{\alpha}|t^n_{\ell}-t^n_{j}|^{\alpha}.
\$ 
Furthermore, since  
$
|t^n_{\ell}-t^n_{j}|\le \omega=(h |\bbf |C_{\alpha}/L)^{1/\alpha}  ,
$  
we have $|\by_{\ell}-\by_{j}| \le 2h\cdot |\bbf|\cdot C_{\alpha}|t_{\ell}-t_{j}|^{\alpha}.$ 
Note that for $j\le \ell \le r $,
\$
I^{i}_{jr}=I^{i}_{j\ell}+I^{i}_{\ell r}+\sum_{k=1}^{h}\bigl( f_{ik}(\by_{\ell})-f_{ik}(\by_j) \bigr)\bigl( x^k_r-x^k_{\ell} \bigr).
\$   
Then we have 
\$
|I^{i}_{jr}| &\le |I^{i}_{j\ell}|+|I^{i}_{\ell r}|+ h|\nabla \bbf| \cdot |\by_{\ell}-\by_j|\cdot C_{\alpha}|t^n_{r}-t^n_{\ell}|^{\alpha}\\
&\le |I^{i}_{j\ell}|+|I^{i}_{\ell r}|+ 2h^{2}|\nabla\bbf |\cdot |\bbf|\cdot C_{\alpha}^{2}\cdot |t^n_r-t^n_j|^{2\alpha}.
\$
Similarly, we have 
\$
|I^{i}_{\ell r}|\le |I^{i}_{\ell,\ell+1}|+|I^{i}_{\ell+1, r}|+ 2h^{2}|\nabla\bbf |\cdot |\bbf|\cdot C_{\alpha}^{2}\cdot |t^n_r-t^n_j|^{2\alpha}.
\$ 
Since $I^{i}_{\ell,\ell+1}=0  $, we get 
\$
|I^{i}_{jr}|\le |I^{i}_{j\ell}|+|I^{i}_{\ell+1, r}|+(2hC_{\alpha})^{2}|\nabla\bbf |\cdot |\bbf|\cdot |t^n_r-t^n_j|^{2\alpha}.
\$   
By applying the inductive hypothesis again, we obtain
\$
|I^{i}_{jr}|&\le L\bigl(|t_j-t_{\ell}|^{2\alpha}+|t_{\ell+1}-t_{r}|^{2\alpha}\bigr) +(2hC_{\alpha})^{2}|\nabla\bbf |\cdot |\bbf| \cdot |t^n_r-t^n_j|^{2\alpha}\\
&\le \bigl(2^{1-2\alpha}L+(2hC_{\alpha})^{2}|\nabla\bbf |\cdot |\bbf|\bigr) \cdot |t^n_r-t^n_j|^{2\alpha}\\
&=L\cdot |t^n_r-t^n_j|^{2\alpha},
\$  
where $  L =(1-2^{1-2\alpha})^{-1}\cdot (2hC_{\alpha})^{2}|\nabla\bbf |\cdot |\bbf|.$ 
Hence, by induction, we finish the proof. Recall the definition of $\bI_{jr}$, we also have that if $|t^n_r-t^n_j|\le \omega$, 
\$
|\by_r-\by_j|\le  |\bI_{jr}|+h|\bbf| C_{\alpha}|t^n_r-t^n_j|^{\alpha} \le \bigl(L+h|\bbf| C_{\alpha}\bigr)\cdot |t^n_r-t^n_j|^{\alpha}.
\$

Now we come back to the proof of Lemma \ref{lem0.1}. For $t^n_{r},t^n_{j} \in [0,1] $ with $|t^n_{r}-t^n_{j}|\le \omega$, we already obtain the conclusion. Otherwise, we can decompose the interval $[t^n_{r},t^n_{j}]$ as 
\$
t^n_{r}=t^n_{k_0}<t^n_{k_1}<t^n_{k_2}<\cdots<t^n_{k_m}=t^n_{j},
\$ 
such that $|t^n_{k_{i+1}}-t^n_{k_i}|\le \omega $ or $k_{i+1}-k_{i}=1$,  $i=0,1,\cdots m-1$. In either case, it is easy to see that 
\$|\by_{k_{i+1}}-\by_{k_{i}}|\le \bigl(L+h|\bbf|C_{\alpha}\bigr)\cdot \bigl|t^n_{k_{i+1}}-t^n_{k_i}\bigr|^{\alpha},\$ 
and hence 
\$
|\by_{j}-\by_{r}|\le \bigl(L+h|\bbf|C_{\alpha}\bigr)\cdot \bigl|t^n_{j}-t^n_{r}\bigr|^{\alpha}\cdot m \le \bigl(L+h|\bbf|C_{\alpha}\bigr)\cdot (1+ \omega^{-1}) \cdot \bigl|t^n_{j}-t^n_{r}\bigr|^{\alpha}.
\$ t
Furthermore, we have 
\$
|\bI_{jr}|&\le |\by_{j}-\by_{r}|+ h|\bbf| C_{\alpha}\bigl|t^n_{j}-t^n_{r}\bigr|^{\alpha} \\
&\le   (2+\omega^{-1})\cdot \big(L+h|\bbf|C_{\alpha} \bigr)\cdot  \bigl|t^n_{j}-t^n_{r}\bigr|^{\alpha}\\
&\le   (2\omega^{-\alpha}+\omega^{-1-\alpha})\cdot \big(L+h|\bbf|C_{\alpha} \bigr)\cdot  \bigl|t^n_{j}-t^n_{r}\bigr|^{2\alpha},
\$   
where we use $|t^n_{j}-t^n_{r}|\ge \omega  $ in   the last inequality. Hence we conclude the proof of Lemma \ref{lem0.1}.

\subsection{Proof of Lemma \ref{lem0.2}}
Based on Taylor's expansion, for $t^n_i,t^n_j \in D_n$, we have
\$
&\Bigl |\bigl[ \bbf(\by(t^n_i))-\bbf(\by(t^n_j))\bigr]-\bigl[  \bbf(\by_n(t^n_i))-\bbf(\by_n(t^n_j)) \bigr] \Bigr |\\
&=\biggl| \int_{0}^1 \nabla \bbf\bigl(\by(t^n_i) +\tau\cdot [ \by_n(t^n_i)-\by (t^n_i) ]\bigr)^{\top} [ \by(t^n_i)-\by_n(t^n_i) ]      \ud \tau- \\
&\qquad \qquad \qquad \qquad  \int_{0}^1 \nabla \bbf\bigl(\by(t^n_j) +\tau\cdot [ \by_n(t^n_j)-\by(t^n_j) ]\bigr)^{\top} [ \by(t^n_j)-\by_n(t^n_j) ]      \ud \tau   \biggr|\\
&\le \biggl| \int_{0}^1 \nabla \bbf\bigl(\by(t^n_i) +\tau\cdot [ \by_n (t^n_i)-\by(t^n_i) ]\bigr)^{\top} \bigl[ (\by(t^n_i)-\by_n(t^n_i)) - ( \by(t^n_j)-\by_n(t^n_j)) \bigr ]           \ud \tau\biggr|+  \\
& \biggl|  \int_{0}^1 \nabla \bbf\bigl(\by(t^n_j) +\tau\cdot [ \by_n(t^n_j)-\by(t^n_j) ] \bigr)-\nabla \bbf\bigl(\by(t^n_i) +\tau\cdot [ \by_n(t^n_i)-\by(t^n_i) ]    \bigr)^{\top} [ \by(t^n_j)-\by_n(t^n_j) ] \ud \tau  \biggr|.
\$ 
We use $A$ and $B$ to denote the two parts in above inequality. Recall the definition of restricted $\alpha$-H\"older norm, we have 
\$
A\le d \cdot |\nabla \bbf|\cdot  H_{\alpha}\bigl(\by-\by_n| D_n\bigr)\cdot |t^n_i-t^n_j|^{\alpha}.
\$ 
For the second term, by mean value theorem, it is easy to have 
\$
B&\le d^2\cdot |\nabla^2 \bbf|\cdot \Bigl( H_{\alpha}(\by| D_n)+H_{\alpha}\bigl (\by_n| D_n\bigl ) \Bigr)\cdot |t^n_i-t^n_j|^{\alpha} \cdot \bigl|  \by(t^n_j)-\by_n(t^n_j) \bigr|\\
&\le d^2\cdot |\nabla^2 \bbf|\cdot ( G_1^* +G_1 )\cdot |t^n_i-t^n_j|^{\alpha} \cdot \bigl|  \by(t^n_j)-\by_n(t^n_j) \bigr|,
\$
where the second inequality follows from Lemma \ref{lem0.0} and \ref{lem0.1}. Note that 
\$
\bigl|  \by(t^n_j)-\by_n(t^n_j) \bigr| &\le |\by(0)-\by_n(0)|+\bigl| (\by(t^n_j)-\by_n(t^n_j)) -(\by(0)-\by_n(0) ) \bigr|\\
&\le |\by(0)-\by_n(0)|+ H_{\alpha}\bigl(\by-\by_n| D_n\bigr)\cdot  |t^n_j|^{\alpha}.
\$
Then we have 
\$
H_{\alpha}\bigl( \bbf(\by)-\bbf(\by_n) | D_n\bigr)& \le \bigl(  d\cdot |\nabla \bbf|+ d^2 \cdot|\nabla^2 \bbf|\cdot ( G_1^* +G_1 ) \bigr)\cdot H_{\alpha}\bigl(\by-\by_n| D_n\bigr)+\\
&\qquad \qquad   d^2 \cdot |\nabla^2 \bbf|\cdot ( G_1^* +G_1 )\cdot |\by(0)-\by_n(0)|,
\$ 
which concludes the proof of Lemma \ref{lem0.2}.
\bibliographystyle{plain}
\bibliography{ref}

\end{document}